\documentclass[12pt]{amsart}
\usepackage{amssymb}

\numberwithin{equation}{section}

\theoremstyle{plain}
\newtheorem{theorem}{Theorem}[section]

\newtheorem{corollary}[theorem]{Corollary}
\newtheorem{proposition}[theorem]{Proposition}
\newtheorem{theorem-definition}[theorem]{Theorem and Definition}
\newtheorem{lemma-definition}[theorem]{Lemma and Definition}
\newtheorem*{theorem*}{theorem}
\newtheorem*{theoremA*}{Theorem A}
\newtheorem*{corollaryB*}{Corollary B}
\newtheorem*{theoremC*}{Theorem C}
\newtheorem*{satellite*}{Satellite to the Main Theorem}

\newtheorem{observation}[theorem]{Observation}
\newtheorem{defthm}[theorem]{Definition and Theorem}
\newtheorem{criterion}[theorem]{Criterion}
\newtheorem{corterm}[theorem]{Corollary and Terminology}

\theoremstyle{definition}
\newtheorem{definition}[theorem]{Definition}

\newtheorem*{definition*}{Definition}
\newtheorem{defobs}[theorem]{Definition and Observation}
\newtheorem*{observation*}{Observation}

\newtheorem{example}[theorem]{Example}

\newtheorem{remarks}[theorem]{Remarks}
\newtheorem*{remark*}{Remark}

\newtheorem*{specialcases*}{Special cases}
\newtheorem*{termcomm*}{Terminology and comments}
\newtheorem*{step1*}{Step 1. Skeleta of modules}
\newtheorem*{step2*}{Step 2. Generic modules for the components of $\layS$}
\newtheorem*{return2.9*}{Return to Example \ref{ex2.9}}
\newtheorem*{param*}{Parametrizing the  $\bd$-dimensional modules with top $T$}
\newtheorem*{param2*}{Parametrizing all $\bd$-dimensional modules}
\newtheorem*{kqmod*}{$KQ$-modules of arbitrary Loewy length}
\newtheorem{algorcomm}[theorem]{Algorithmic Comments}
\newtheorem*{return8.12*}{Return to Example \ref{ex8.12}}
\newtheorem*{prog1*}{(1)}
\newtheorem*{prog2*}{(2)}

\input xy \xyoption{matrix} \xyoption{arrow}
          \xyoption{curve}  \xyoption{frame}
\def\edge{\ar@{-}}
\def\dttdar{\ar@{.>}}
\def\drbl{\save+<0ex,-2ex> \drop{\bullet} \restore}

\def\horizpool#1{  \save[0,0]+(0,3);[0,0]+(0,3) **\crv{~*=<2.5pt>{.} [0,#1]+(0,3) &[0,#1]+(3,3) &[0,#1]+(3,0) &[0,#1]+(3,-3) &[0,#1]+(0,-3) &[0,0]+(0,-3) &[0,0]+(-3,-3) &[0,0]+(-3,0) &[0,0]+(-3,3)}\restore }
\def\dashedge{\ar@{--}}

\def\dshdar{\ar@{-->}}
\def\dropdown#1{\save+<0ex,-4ex> \drop{#1} \restore}
\def\dropup#1{\save+<0ex,4ex> \drop{#1} \restore}
\def\dropvert#1#2{\save+<0ex,#1ex>\drop{#2}\restore}
\def\fatedge#1{\edge@<0.025pc>[#1] \edge@<-0.025pc>[#1] \edge@<0.05pc>[#1]  \edge@<-0.05pc>[#1] \edge[#1]}

\def\mbull{\ar@{}[r]|{\bullet}}
\def\dbull{*{\bullet}+0}

\def\vsubseteq{\hbox{$\bigcup$\kern0.1em\raise0.05ex\hbox{$\textstyle|$}}}
\def\smvsubseteq{\hbox{$\ssize\bigcup$\kern0.03em\raise0.05ex\hbox{$\ssize|$}}}

\def\rtri{\,\lower.5ex\vbox{\hrule width.5pt height2ex} \kern-.1ex\diagdown \kern-2ex\lower.5ex\vbox{\hrule width1.9ex height.5pt}}
\def\rtriangle{|\kern-.2em\backslash \kern-1ex\lower.6ex\vbox{\hrule width.35em height.5pt}}

\def\la{{\Lambda}}
\DeclareMathOperator \trunc {trunc}
\def\latrunc{\la_{\trunc}}

\def\lamod{\Lambda\mbox{\rm-mod}}

\def\AA{{\mathbb A}}
\def\CC{{\mathbb C}}
\def\DD{{\mathbb D}}
\def\EE{{\mathbb E}}
\def\PP{{\mathbb P}}
\def\SS{{\mathbb S}}
\def\ZZ{{\mathbb Z}}
\def\NN{{\mathbb N}}
\def\QQ{{\mathbb Q}}

\DeclareMathOperator \len {length}

\DeclareMathOperator \aut {Aut}
\DeclareMathOperator \Aut {Aut}
\DeclareMathOperator \Top {top}
\DeclareMathOperator \soc {soc}

\DeclareMathOperator \rank {rank}

\DeclareMathOperator \ann {ann}
\DeclareMathOperator \nullity {nullity}

\DeclareMathOperator \pdim {p\,dim}
\DeclareMathOperator \idim {i\, dim}

\DeclareMathOperator \strt {start}

\DeclareMathOperator \term {end}

\DeclareMathOperator \GL {GL}

\DeclareMathOperator \Img {Im}
\DeclareMathOperator \Seq {\mathbf {Seq}}

\DeclareMathOperator \Ext {Ext}
\DeclareMathOperator \ext {ext}

\DeclareMathOperator \End {End}

\DeclareMathOperator \Filt {\mathbf{Filt}}

\DeclareMathOperator \udim {\underline{dim}}

\def\A{{\mathcal A}} 
\def\B{{\mathcal B}} 
\def\C{{\mathcal C}}
\def\D{{\mathcal D}}
\def\E{{\mathcal E}}

\def\G{{\mathcal G}}

\def\M{{\mathcal M}}

\def\U{{\mathcal U}}
\def\S{{\sigma}}

\def\bA{\mathbf {A}}
\def\bB{\mathbf {B}}

\def\bP{\mathbf {P}}

\def\bd{\mathbf {d}}

\def\bt{\mathbf {t}}

\def\Phat{\widehat{P}}
\def\Qhat{\widehat{Q}}

\def\SStilde{\widetilde{\SS}}

\def\Gtilde{\widetilde{G}}

\def\bP{\mathbf{P}}
\def\SShat{\widehat{\SS}}
\def\SStilde{\widetilde{\SS}}
\def\ehat{\widehat{e}}

\def\uhat{\widehat{u}}
\def\bdhat{\widehat{\bd}}

\def\lahat{\widehat{\la}}

\def\alphahat{\widehat{\alpha}}

\def\Chat{\widehat{C}}

\def\Jhat{\widehat{J}}

\def\Mhat{\widehat{M}}

\def\Phat{\widehat{P}}
\def\Qhat{\widehat{Q}}

\def\That{\widehat{T}}

\def\uhat{\widehat{u}}

\def\bz{\mathbf{z}}

\DeclareMathOperator \rep {\mathbf{Rep}}
\DeclareMathOperator \Rep {\mathbf{Rep}}
\DeclareMathOperator \layS {\Rep(\S)}
\DeclareMathOperator \laySS {\Rep \SS}
\newcommand \modlad {\Rep_\bd(\Lambda)}
\DeclareMathOperator \modlasd {{\Rep}_{d}(\Lambda)}

\DeclareMathOperator \Mod {\Rep}

\DeclareMathOperator \Hom {Hom}

\DeclareMathOperator \Gr {Gr}
\DeclareMathOperator \grasstbd {\mathfrak{Grass}^T_{\mathbf{d}}}
\def\autlap{\aut_\Lambda(P)}

\DeclareMathOperator \Gal {Gal}
\DeclareMathOperator \grass {\mathfrak{Grass}}
\DeclareMathOperator \grassS {\grass(\S)}

\DeclareMathOperator \grassSS {\grass \SS}

 \DeclareMathOperator \biggrass {GRASS}

\DeclareMathOperator \GRASS {GRASS}

\def\grassbd{\GRASS_\bd(\Lambda)}

\def\biggrassSS{\GRASS(\SS)}

\def\Gabull{\Gamma_\bullet}
\def\Znn{\ZZ_{\ge0}}
\def\hatS{\widehat{S}}

\title{Understanding finite dimensional representations generically}
\author[K. R. Goodearl]{K. R. Goodearl}
\address{
Department of Mathematics \\
University of California\\
Santa Barbara, CA 93106 \\
U.S.A.
}
\email{goodearl@math.ucsb.edu}
\author[B. Huisgen-Zimmermann]{B. Huisgen-Zimmermann}
\address{
Department of Mathematics \\
University of California\\
Santa Barbara, CA 93106 \\
U.S.A.
}
\email{birge@math.ucsb.edu}

\begin{document}

\dedicatory{Dedicated to Dave Benson on the occasion of his sixtieth birthday} 




\maketitle

\section{Introduction and conventions}
\label{sec1}

The complex of problems addressed in this survey aims at an analysis of the ``bulk", in a geometric sense, of finite dimensional representations of a finite dimensional algebra $\la$.  We assume the base field $K$ to be algebraically closed and $\la$ to be basic finite dimensional over $K$, whence we do not lose generality in identifying $\la$ with a path algebra modulo relations: $ \la = KQ/I$, for some quiver $Q$.   The pivotal problems originated with two groundbreaking papers of Kac in the early 1980s (\cite{KacI}, 1980, and \cite{KacII}, 1982); both focus on hereditary algebras, i.e., algebras of the form $\la = KQ$.  We excerpt a quote from the introduction to the 1982 article:

``The problem [of classifying all representations of $\la$] seems to be hopeless in general.  According to general principles of invariant theory, it is natural to try to solve a simpler problem:  Classifying the `generic' representations of a given dimension [vector] $\bd$."  

The last 35 years have shown that, while this problem is certainly \emph{simpler} than establishing an all-encompassing classification of the $\bd$-dimensional $\la$-modules for arbitrary $\bd$, it is by no means \emph{simple}.  Nor should it be viewed as an isolated problem of the type expected to find a useful solution in one fell swoop.  Rather, it constitutes a program, to be pursued long-term.  This is all the more true as the task turns significantly more intricate when one moves beyond the case $\la = KQ$.

In its strongest form, Kac's challenge calls for a rigorous classification of the modules in a dense open subset $\U$ of any of the standard parametrizing varieties $\Rep_\bd(KQ)$.  The approach that first comes to mind remains in the geometric context:  As such, it calls for specification of a nonempty open subset $\U$, which is stable under the $\GL(\bd)$-action and possesses a geometric quotient with respect to this action, such that $\U/ \GL(\bd)$ is a fine moduli space for the isomorphism classes of representations in $\U$; in intuitive terms, the task involves pinning down normal forms for the modules in $\U$ which are in ``natural" bijection with the points of $\,\U/\GL(\bd)$.  Below, we will briefly comment on such an ambitious endeavor in the more general context. The core of our overview will focus on a more modest interpretation of Kac's prompt, however.  In case $\la$ is hereditary, it calls for a list  --  representative in a sense to be spelled out  --   of ``essential", ``generic" isomorphism invariants of the $\bd$-dimensional $KQ$-modules; that is, of invariants $\bullet$ preserved by Morita self-equivalences of $KQ\text{-mod}$ (\emph{essential}) and $\bullet$ shared by the modules in a dense open subset of $\Rep_\bd(KQ)$ (\emph{generic}).  For $\la = KQ$, this is a meaningful goal, since $\Rep_\bd(KQ)$ is an affine space.  In particular, due to irreducibility, one targets  module invariants which are constant  on suitable dense open subsets of $\Rep_\bd(KQ)$; see Section \ref{sec2} for prototypes.  However, in extending Kac's idea beyond the hereditary case, one needs to take into account that $\modlad$ consists of a plethora of irreducible components in general, and that hardly any relevant condition imposed on the corresponding modules can be expected to hold across  dense subsets of all components.  Hence the quest now targets the generic representations in each of the individual components,  leading to the following program:  

\begin{prog1*}  Find the irreducible components of $\modlad$ in a representation-theoretic format, that is, in terms of module invariants which cut the components out of the  parametrizing variety.  
\end{prog1*}

\begin{prog2*} For each component $\C$ of $\modlad$, determine the \emph{essential} \emph{generic} properties of the modules ``in" $\C$.  (As above, \emph{essential} means invariant under Morita self-equivalences of $\lamod$.  Moreover, recall that a module property is \emph{generic for $\C$} if it holds for all modules in a dense open subset of $\C$.)
\end{prog2*}

The two points of the program are strongly interconnected. After all, representation-theoretically characterizing the components of $\modlad$ typically amounts to pinning down families of generic invariants which separate them, combined with an understanding of the families of values that occur.  

As for the more taxing goal we alluded to, that of rigorously classifying the modules in a suitable dense open subset of any of the components of $\modlad$:  It is not excluded from the theoretically feasible.  Indeed, a result of Rosenlicht \cite{Ros} guarantees that any irreducible variety $X$ which carries a morphic action by an algebraic group $\G$ contains a $\G$-stable dense open subset $\U$ which admits a geometric quotient modulo $\G$.  Necessarily, the dimension of the quotient $\U/\G$ equals 
$$\mu(X) := \dim X - \max\{ \dim \G.x \mid x \in X\},$$
the \emph{generic number of parameters of $X$}.  (Suppose $X$ is a component of $ \modlad$ and $\G = \GL(\bd)$. Loosely speaking, $\mu(X)$ is then the number of independent parameters appearing in the aforementioned normal forms for the modules in $\U$.)  However, this existence statement, applied to a component of $\modlad$, has limited value towards the algebraic understanding of the representations of $\la$, unless one is able to specify an appropriate open set $\U$ in  representation-theoretic terms and relate the structure of the encoded modules to the points of the geometric quotient $\U/ \GL(\bd)$.  Barring special cases, such an objective does not appear within reach at the moment.

Guideline through the paper:
We begin with a brief discussion of  generic module properties in Section \ref{sec2}, followed by a cursory overview of results to date in Section \ref{sec3}.  The information pertaining to the individual points of the overview will then be refined and supplemented in Sections \ref{sec4}--\ref{sec9} according to the table of contents at the end of this section.

\subsection*{Further conventions} Throughout, $J$ denotes the Jacobson radical of $\la$ and $L+1$ is an upper bound for the Loewy length of $\la$, i.e., $J^{L+1} = 0$. Let $e_1,\dots, e_n$ denote the distinct vertices of $Q$; we identify them with both the paths of length zero in $KQ$ and the corresponding primitive idempotents in $\la$. The simple module corresponding to $e_i$, namely $\la e_i/Je_i$, will be denoted by $S_i$. Paths in $KQ$ and (their images) in $\la$ are to be composed like functions, i.e., $pq$ stands for ``$p$ after $q$" in case $\strt(p) = \term(q)$, while $pq=0$ otherwise.

A \emph{dimension vector} (for $Q$ or $\la$) is any vector $\bd \in \Znn^n$, and the dimension vector of a (finitely generated) $\la$-module $M$ is the vector
$$\udim M := (\dim e_1M, \dim e_2M, \dots, \dim e_nM),$$
whose entries give the multiplicities of the simples $S_i$ as composition factors of $M$. We denote by $\modlad$ the standard affine variety parametrizing the $\bd$-dimensional $\la$-modules; it consists of the tuples $(x_\alpha)_{\alpha \in Q_1}$ in $\prod_{\alpha\in Q_1} \Hom_K\bigl( K^{d_{\strt(\alpha)}}, K^{d_{\term(\alpha)}} \bigr)$ such that the $x_\alpha$ satisfy all the relations in the ideal $I$. (As usual, $Q_1$ denotes the set of arrows of $Q$.) We write $M_x$ for the module corresponding to a point $x\in \modlad$. The sets of points in $\modlad$ cut out by  the isomorphism classes of $\bd$-dimensional $\la$-modules are precisely the orbits under the natural conjugation action of $\GL(\bd) := \prod_{1\le i\le n} \GL_{d_i}(K)$ on $\modlad$. 

A \emph{top element} of a module $M$ is an element $z \in M \setminus JM$ which is \emph{normed} by some primitive idempotent $e_i$, meaning that $z = e_i z$; in particular, $\la  (z + JM) \cong S_i$ in $M/JM$ in this case.  A \emph{full set of top elements} for $M$ is a set of top elements which induces a $K$-basis for $M/JM$.

\subsection*{Graphing} We use (layered and labeled) graphs to profile the structure of a module (or class of modules) $M$; the graphs used here are slightly simplified variants of those appearing in \cite{irredcompII} and \cite{GHZ}. They emphasize the radical layering $(J^l M/ J^{l+1} M)_{0 \le l \le L}$, pivotal in identifying the irreducible components of $\modlad$.  We first sketch the most straightforward type of graph; it is limited with regard to the linear dependencies it permits to encode. 
The vertices in layer $l$ correspond to a full set of top elements of $J^lM$, i.e., they represent the simple direct summands of the $l$-th radical layer $J^lM/ J^{l+1}M$; the label of a vertex coincides with that of the norming idempotent.  For $\alpha: e_i \rightarrow e_j$  in $Q_1$, an edge labeled $\alpha$ from a vertex $i$ in some layer $l$ to a vertex $j$ in a lower layer (= layer of higher index) communicates the action of $\alpha$ on the corresponding top element of $J^l M$ up to a scalar factor from $K^*$. For example, a graph of the form 
$$\xymatrixrowsep{1.7pc} \xymatrixcolsep{1pc}
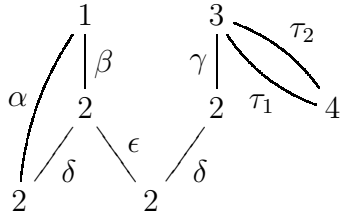
\xymatrix{
&1 \edge@/_/[ddl]_{\alpha} \edge[d]^{\beta} &&3 \edge[d]_{\gamma} \edge@/_/[drr]_(0.6){\tau_1} \edge@/^/[drr]^{\tau_2}  \\
&2 \edge[dl]^{\delta} \edge[dr]^{\epsilon} &&2 \edge[dl]^{\delta} &&4  \\
2 &&2 
}$$
stands for a family of modules $M$ sharing the following properties:  $M/JM \cong S_1\oplus S_3$, \  $JM/J^2M \cong S_2^2 \oplus S_4$, and $JM \cong S_2^2$. Moreover, the graph tells us that a full set of top elements $z_1$, $z_2$ of $M$ with $z_1= e_1z_1$ and $z_2 = e_3z_2$ may be chosen such that
\begin{align*}
JM &= \la \beta z_1 + \la \gamma z_2 + \la \tau_1 z_2 = \la \beta z_1 + \la \gamma z_2 + \la \tau_2 z_2\ \ \ \text{and}  \\
J^2M &= \la \alpha z_1 + \la \epsilon \beta z_1 =  \la \alpha z_1 + \la \delta \gamma z_2 = \la \delta \beta z_1 + \la \epsilon \beta z_1 =  \la \delta \beta z_1 + \la \delta \gamma z_2  \, .
\end{align*}
More specifically, the graph conveys that $\alpha z_1$ is a nonzero scalar multiple of $\delta \beta z_1$ and $\epsilon \beta z_1 \in K^* \delta \gamma z_2$, while $\tau_2 z_2 \in K^*\tau_1 z_2$ and $\epsilon \gamma z_2 = 0$.

If we wish to encode linear dependencies of $3$ or more displayed elements labeled by the same simple $S_i$, the number of vertices  $i$ will in general be higher than $\dim e_i M$.  We use variants of the above types of graphs, which allow for  ``pooling" of vertices, such as:
$$\xymatrixrowsep{1.7pc} \xymatrixcolsep{1pc}
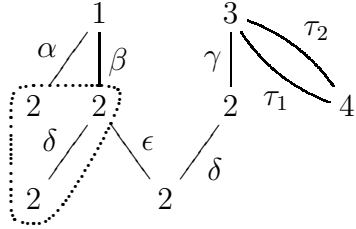
\xymatrix{
&1 \edge[dl]_{\alpha} \edge[d]^{\beta} &&3 \edge[d]_{\gamma} \edge@/_/[drr]_(0.6){\tau_1} \edge@/^/[drr]^{\tau_2}  \\
2 \save[0,0]+(0,3);[0,0]+(0,3) **\crv{~*=<2.5pt>{.} [0,0]+(3,3) &[0,1]+(0,3) &[0,1]+(3,3) &[0,1]+(3,0) &[1,0]+(3,-3) &[1,0]+(0,-3) &[1,0]+(-3,-3) &[1,0]+(-3,0) &[0,0]+(-3,0) &[0,0]+(-3,3)} \restore &2 \edge[dl]_{\delta} \edge[dr]^{\epsilon} &&2 \edge[dl]^{\delta} &&4  \\
2 &&2 
}$$
\noindent  Any module $N$ in the family represented by this graph has the same radical layering as $M$.  Indeed, the dotted pool communicates the fact that $\alpha z_1$, $\beta z_1$, $\delta \beta z_1$ are linearly dependent, while any two of these elements are linearly independent.  In other words, $\alpha z_1= x_1 \beta z_1 + x_2 \delta \beta z_1$ in $N$ for suitable $x_i \in K^*$, whence we find that $\alpha z_1$ and $\beta z_1$ only contribute one copy of $S_2$ to $JN/J^2N$.  In light of $\delta^2 \beta z_1 = 0$ (as communicated by the graph), we moreover glean $\delta \alpha z_1 = x_1 \delta \beta z_1$.  In particular, existence of a $\la$-module $N$ satisfying the equality $\alpha z_1= x_1 \beta z_1 + x_2 \delta \beta z_1$ implies that $\delta \alpha \neq c\, \delta \beta$ in $\la$ for any constant $c$ different from $ x_1$. 

In spite of the fact that $\delta \alpha z_1 \in K^* \delta \beta z_1$, we did not include an edge labeled $\delta$ between the two left-most vertices `2' of the graph.  We only insist on showing edges that carry irredundant information. 

Finally, we observe that $N \not\cong M$ for any choice of $M$ and $N$ in the two depicted families, since $\alpha M \subseteq J^2M$ while $\alpha N \nsubseteq J^2N$.
\bigskip

\centerline{TABLE OF CONTENTS}

\begin{enumerate}
\item[\ref{sec1}.] Introduction and conventions
\item[\ref{sec2}.]  Generic module properties and semicontinuous maps on $\modlad$
\item[] \ref{sec2}.A. Semicontinuous invariants 
\item[] \ref{sec2}.B. Detection and separation of irreducible components via upper semicontinuous maps 
\item[] \ref{sec2}.C.  A crucial generic invariant which fails to be semicontinuous on $\modlad$
\item[] \ref{sec2}.D.  A running example
\item[\ref{sec3}.] In a nutshell: Results to date
\item[] \ref{sec3}.A.  Hereditary algebras, $\la = KQ$
\item[] \ref{sec3}.B. Tame non-hereditary algebras
\item[] \ref{sec3}.C.  Wild non-hereditary algebras
\item[\ref{sec4}.] General facts about components and generic properties of their modules  
\item[] \ref{sec4}.A. Canonical decompositions of the irreducible components of $\modlad$
\item[] \ref{sec4}.B. Where to look for generic properties: Generic modules for the  components   
\item[\ref{sec5}.] More detail on the hereditary case
\item[\ref{sec6}.]  More detail on the tame non-hereditary  case
\item[\ref{sec7}.] Projective parametrizing varieties
\item[] \ref{sec7}.A. The ``small" projective parametrizing varieties $\grasstbd$ and $\grassSS$
\item[] \ref{sec7}.B.  The ``big" projective parametrizing varieties $\grassbd$ and $\biggrassSS$
\item[\ref{sec8}.]  The wild case:  Focus on truncated path algebras
\item[] \ref{sec8}.A. Realizability criterion and generic socles
\item[] \ref{sec8}.B.  The most complete generic picture:  $J^2 = 0$
\item[] \ref{sec8}.C.  Local algebras
\item[] \ref{sec8}.D. Algebras based on acyclic quivers
\item[] \ref{sec8}.E. The general truncated case
\item[\ref{sec9}.]  Beyond truncated path algebras
\item[] \ref{sec9}.A.  A finite set of irreducible subvarieties of $\modlad$ including all components
\item[] \ref{sec9}.B. Facts which carry over from truncated to general algebras
\item[] \ref{sec9}.C. Interplay between $\modlad$ and $\Rep_\bd(\latrunc)$
\item[] \ref{sec9}.D.  Illustration
\end{enumerate}

\section{Generic module properties and semicontinuous maps on $\modlad$}
\label{sec2}

By a \emph{generic} (\emph{module}) \emph{property}, not tied to any particular component, we will mean any property which, for arbitrary $\bd$ and any irreducible component $\C$ of $\modlad$,  is constant on some dense open subset of $\C$.  

There are two fundamentally different types of generic module properties, distinguished by whether or not they result from semicontinuous maps on $\modlad$.  Accordingly, we split the discussion of generic invariants into two cases.  The data associated with semicontinuous maps turn out to be particularly useful towards the detection of components, in a sense to be made precise in \ref{sec2}.B. 

\subsection*{\ref{sec2}.A. Semicontinuous invariants}

\begin{definition} \label{def2.1} Suppose $X$ is a topological space and $(\A, \le )$ a poset.  For $a \in \A$, we denote by $[a, \infty)$ the set $\{b \in \A \mid b \ge a\}$; the sets $(a, \infty)$, $(- \infty, a]$ and  $(- \infty , a)$ are defined analogously.

A map 
$f: X \longrightarrow \A$ is called \emph{upper semicontinuous} if, for every element $a \in \A$, the pre-image of $[a, \infty)$ under $f$ is closed in $X$.  
\end{definition}

We start with a few well known examples of upper semicontinuous maps on $X = \modlad$  with the Zariski topology.  Further examples will be encountered along the way.
Many module invariants taking numerical values are well known to yield upper semicontinuous maps.  For any fixed $N \in \lamod$, the maps $x \mapsto \dim \Hom_\la(M_x, N)$ and $x \mapsto \dim \Hom_\la(N, M_x)$, $x \mapsto \dim \Ext^1_\la(M_x, N)$ and $x \mapsto  \dim \Ext^1_\la(N,M_x)$ are examples, as is $x \mapsto \dim \End_\la(M_x)$; for $\Ext^1$ and $\End$, see \cite[Lemma 4.3]{CBS}.  Additional examples are the homological dimensions $\modlad \rightarrow \ZZ \cup \{\infty\}$, $x \mapsto \pdim M_x$ and $x \mapsto \idim M_x$ (see \cite[Theorem 12.61]{JeLe} or \cite[Lemma 2.1]{MHSa}).  Moreover, for any path $p$ in $KQ \setminus  I$, the map $x \mapsto \nullity_p M_x$ is upper semicontinuous, where $\nullity_p M_x$ is the nullity of the $K$-linear map $M_x \rightarrow M_x, \ m \mapsto p\, m$. 

Of course, numerical lower semicontinuous maps may be converted into upper semicontinuous ones by way of a factor $-1$.
The following functions are less standard.

\begin{definition} \label{def2.2} A \emph{semisimple sequence} is a sequence $\SS = (\SS_0,\SS_1,\dots,\SS_L)$ whose entries are semisimple $\la$-modules. (Recall that $L$ with $J^{L+1} = 0$ is fixed.) The dimension vector of such a sequence is $\udim \SS := \sum_{0\le l\le L} \udim \SS_l$. For any dimension vector $\bd$, we write $\Seq(\bd)$ for the set of $\bd$-dimensional semisimple sequences. This set is partially ordered by the \emph{dominance order}, defined as follows:
$$(\SS_0,\dots,\SS_L) \le (\SS'_0,\dots,\SS'_L) \ \ \ \iff \ \ \ \bigoplus_{0 \le j \le l} \SS_j \subseteq \bigoplus_{0 \le j \le l} \SS'_j \ \ \text{for}\  \ l \in \{0, 1, \dots, L\}.$$
Isomorphic semisimple modules will be identified; hence the above inclusion amounts to `$\le$' for the corresponding dimension vectors. 

The \emph{radical layering} and \emph{socle layering} of a $\la$-module $M$ are the semisimple sequences
\begin{align*}
\SS(M) &:= (M/JM, JM/J^2M, \dots, J^LM) \\
\SS^*(M) &:= (\soc_0(M), \soc_1(M)/\soc_0(M), \dots, M/\soc_L(M)),
\end{align*}
where $\soc_0(M) = \soc(M)$ and $\soc_{l+1}(M)/\soc_l(M) = \soc(M/\soc_l(M))$. For any semisimple sequence $\SS$ with $\udim \SS = \bd$, the following is a  locally closed subvariety of $\modlad$:
$$\laySS := \{ x \in \modlad \mid \SS(M_x) = \SS \}.$$
\end{definition}

\begin{proposition} \label{prop2.3} {\rm \cite[Observation 2.10]{irredcompI}} The maps $\modlad \rightarrow \Seq(\bd)$ defined by radical and socle layerings,  $x \mapsto \SS(M_x)$ and $x \mapsto \SS^*(M_x)$, are upper semicontinuous.
\end{proposition}

\subsection*{\ref{sec2}.B. Detection and separation of irreducible components via upper semicontinuous maps}

The upcoming observation provides the link to the component problem. Note that the hypotheses are satisfied for all of the examples listed above.

\begin{observation} \label{obs2.4} {\rm\cite[Observation 2.7]{irredcompI}} Let $\A$ be a poset, $X$ a topological space, and $f: X \rightarrow  \A$ an upper semicontinuous map whose image is well partially ordered {\rm (meaning that $\Img(f)$ does not contain any infinite strictly descending chain and every nonempty subset has only finitely many minimal elements)}.

Then the pre-images $f^{-1}\bigl(( - \infty, a) \bigr)$ and $f^{-1}\bigl(( - \infty, a] \bigr)$ for $a \in \A$ are open in $X$.  In particular, given any irreducible subset $\U$ of $X$, the restriction of $f$ to $\U$ is generically constant, and the generic value of $f$ on $\,\U$ is \ 
$\min\{f(x) \mid x \in \U\}$.
\end{observation}

\begin{specialcases*} Let $\C$ and $\C'$ be irreducible components of $\modlad$ and $\Rep_{\bd'}(\la)$, respectively.

$\bullet$ The generic value of the map $\C \times \C' \rightarrow \Znn$, $(x, x') \mapsto \dim \Ext^1_\la(M_x, M_{x'})$, namely the minimum of the values attained, will be denoted by $\ext(\C, \C')$; 

$\bullet$ that of the map $\C \times \C' \rightarrow \Znn$, $(x, x') \mapsto \dim \Hom_\la(M_x, M_{x'})$ will be denoted by \, hom$(\C, \C')$;

$\bullet$ that of the map $\C \rightarrow \Znn$, $x \mapsto \dim \End_\la(M_x)$ will be denoted by $\term(\C)$.  
\smallskip

Note that $\term(\C)$ is decisive towards determining the generic number of parameters of $\C$, in that $\mu(\C) = \dim\C - \dim \GL(\bd) + \term(\C)$. (Indeed, $\dim \GL(\bd) - \dim \GL(\bd).x = \dim \End_\la(M_x)$ for $x\in\C$; see, e.g., \cite[p.~17]{CB}.) Moreover, it is clear that $\term(\C) = 1$ implies generic indecomposability of the modules in $\C$.  The converse fails in general; think of $\la = K[X]/ (X^2)$ and $\bd = d = 2$, for instance.
\end{specialcases*}  

\begin{definition} \label{def2.5} Let $f: X \rightarrow \A$ be as in Observation \ref{obs2.4}.  We say that $f$ \emph{detects irreducible components} provided that, for each irreducible component $\C$ of $X$, the generic value of $f$ on $\C$ is minimal in $\Img(f)$; equivalently, $\C \cap f^{-1}(a) \ne \varnothing$ for some minimal element $a \in \Img(f)$. We say that $f$ \emph{detects and separates irreducible components} if, additionally, $f^{-1}(a)$ is irreducible for every minimal element $a \in \Img(f)$. 
\end{definition}

\begin{example} \label{ex2.6} \cite[Example 2.11]{irredcompI}
Let $\la = KQ/ \langle \beta_1\alpha_2, \, \beta_2\alpha_1 \rangle$ and $\bd = (1,1,1)$, where $Q$ is the quiver 
$$\xymatrixrowsep{4pc}\xymatrixcolsep{4pc}
\xymatrix{
1 \ar@/_/[r]_{\alpha_2} \ar@/^/[r]^{\alpha_1} &2  \ar@/_/[r]_{\beta_2} \ar@/^/[r]^{\beta_1} &3
}$$
Then $\modlad$ has two irreducible components, whose modules generically have the forms
$$\xymatrixrowsep{1.25pc}\xymatrixcolsep{3pc}
\xymatrix{
1 \edge[d]_{\alpha_1} &&1 \edge[d]_{\alpha_2}   \\
2 \edge[d]_{\beta_1} &\txt{and} &2 \edge[d]_{\beta_2}  \\
3 &&3
}$$
The map $x \mapsto (\SS(M_x), \SS^*(M_x))$ detects, but does not separate them.  On the other hand, the pair of path nullities $x \mapsto \bigl(\nullity_{\beta_1 \alpha_1} M_x, \,\nullity_{\beta_2 \alpha_2} M_x \bigr)$ detects \emph{and} separates the components. 
\end{example}

\subsection*{\ref{sec2}.C.  A crucial generic invariant which fails to be semicontinuous on $\modlad$ }

In \cite{KacI, KacII}, Kac found the numerical invariants governing indecomposable decompositions of modules over path algebras to be generically constant, an observation carried over to general $\la$ by de la Pe\~na \cite{dlP}.  We decompose any $M \in \lamod$ in the form 
$$M = \bigoplus_{1 \le u \le s(M)} M_u \, ,$$
where each $M_u$ is indecomposable. 

\begin{proposition} \label{prop2.7} Let $\la$ and $\bd$ be arbitrary and $\C$ an irreducible component of $\modlad$. Then $s(M)$ and the family $\bigl( \udim M_u \bigr)_{u \le s(M)}$ are generically constant $($the latter up to order$)$ as $M$ traces $\C$.  
\end{proposition}

\begin{proof} See \cite[Section 2.8(a)]{KacI}, \cite[Proposition 3]{KacII}, \cite{dlP}, \cite[Theorem 1.1]{CBS}.
\end{proof}

\begin{termcomm*}  First suppose that $\la = KQ$.  In light of irreduciblity of $\Rep_\bd(KQ)$, any dimension vector $\bd$ can thus  be written in the form $\bd = \bd^{(1)} + \dots + \bd^{(s)}$ such that for all points $x$ in a suitable dense open subset of $\Rep_\bd(KQ)$, we have $M_x = \bigoplus_{1 \le u \le s} M_x^{(u)}$ with $M_x^{(u)}$ indecomposable and $\udim M_x^{(u)} = \bd^{(u)}$.  Kac dubbed this sum presentation of $\bd$ the \emph{canonical decomposition} of $\bd$; it is unique up to order of the summands.  When $\modlad$ fails to be irreducible, Proposition \ref{prop2.7} guarantees an analogous decomposition of $\bd$ for each irreducible component of $\C$.  We refer to it as the \emph{Kac decomposition} of $\bd$ \emph{relative to} $\C$, in order to reserve the attribute ``canonical" for a more informative decomposition of $\C$ established by Crawley-Boevey and Schr\"oer in \cite{CBS}; see Section \ref{sec4}.A for detail.   The Kac decompositions of $\bd$ relative to distinct components of $\modlad$ differ in general, as will shortly be illustrated.    

We follow with an example attesting to the fact that the dependence on $x \in \modlad$ of the number $s(x) = s(M_x)$ of indecomposable direct summands of $M_x$ fails to be semicontinuous.  In other words, this number does  not belong to the class of invariants discussed in \ref{sec2}.A.
\end{termcomm*}

\begin{example} \label{ex2.8} Let $\la = KQ$ be the Kronecker algebra, i.e., $Q= \xymatrixcolsep{2pc} \xymatrix{ 1 \ar@/^/[r]^{\alpha} \ar@/_/[r]_{\beta} &2 }$, and $\bd = (2,2)$. Then the generic value of $s(M)$ on $\modlad$ is $2$, corresponding to a generic decomposition of the form
$$\xymatrixrowsep{0.25pc} \xymatrixcolsep{1.5pc}
\xymatrix{
1 \edge@/_/[dd]_{\alpha} \edge@/^/[dd]^{\beta} &&1 \edge@/_/[dd]_{\alpha} \edge@/^/[dd]^{\beta}  \\
&\bigoplus  \\
2 &&2
}$$
On the other hand, all band or string modules with dimension vector $\bd$ are indecomposable, such as 
$$\xymatrixrowsep{1.5pc} \xymatrixcolsep{3pc}
\xymatrix{
1 \edge[d]_{\alpha} \edge[dr]^{\beta} &1 \edge[d]^{\alpha}  \\
2 &2
}$$
In particular, the generic value of $s(M)$ on $\modlad$ fails to be the minimal one, whence $x \mapsto s(M_x)$ is not upper semicontinuous on $\modlad$. On the other hand, the generic value of $s(M)$ on $\modlad$ is clearly smaller than the maximal value, namely $|\bd|$, whence lower semicontinuity is ruled out as well.
\end{example}

\subsection*{\ref{sec2}.D. A running example}

The following example illustrates the concepts of the section.  It will be revisited repeatedly.

\begin{example} \label{ex2.9}  Let $\la = \CC Q / I$, where $Q$ is the quiver below, and $I$ is generated by $\beta_i \alpha_j$  for $i \ne j$ together with $\alpha_1 \beta_2$ and all paths of length $4$.
$$\xymatrixrowsep{1.5pc}\xymatrixcolsep{6pc}
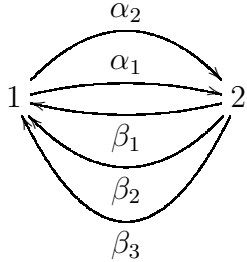
\xymatrix{
1 \ar@/^/[r]^{\alpha_1}  \ar@/^5ex/[r]^{\alpha_2} &2
\ar@/^/[l]^{\beta_1}  \ar@/^5ex/[l]^{\beta_2} \ar@/^9ex/[l]^{\beta_3}
}$$

If $\bd = (1,1)$, then $\modlad$ has $2$ irreducible components.  Generically, their modules have the following graphs, respectively.
$$\xymatrixrowsep{2pc} \xymatrixcolsep{2pc}
\xymatrix{
1 \edge@/_/[d]_{\alpha_1} \edge@/^/[d]^{\alpha_2} &&&2 \edge@/_1pc/[d]_{\beta_1} \edge[d]^{\beta_2} \edge@/^2pc/[d]^{\beta_3}  \\ 
2 &&&1
}$$

For $\bd = (2,1)$, there are $4$ irreducible components in $\modlad$, which we again communicate by way of generic graphs of their modules:
$$\xymatrixrowsep{1.5pc}\xymatrixcolsep{1pc}
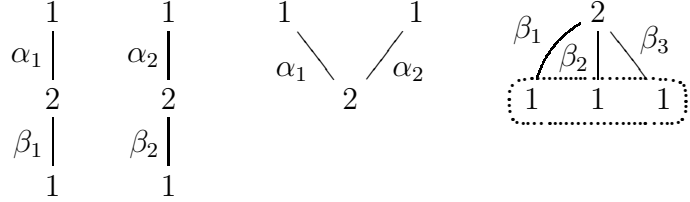
\xymatrix{
1 \edge[d]_{\alpha_1} &&1 \edge[d]_{\alpha_2} &&1 \edge[dr]_{\alpha_1} &&1 \edge[dl]^{\alpha_2} &&&2 \edge@/_/[dl]_{\beta_1} \edge[d]_{\beta_2} \edge[dr]^{\beta_3}  \\
2 \edge[d]_{\beta_1} &&2 \edge[d]_{\beta_2} &&&2 &&&1 \horizpool2 &1 &1  \\
1 &&1
}$$
The generic radical layerings of the modules in these components are $(S_1, S_2, S_1, 0)$ for the first two components, $(S_1^2, S_2, 0, 0)$ for the third, and $(S_2, S_1^2, 0, 0)$ for the last.  The generic socle layerings may be read off the graphs equally readily in this case. 

If $\bd= (2,2)$, then $\modlad$ has $8$ irreducible components (see Section \ref{sec9}.D).  Generically, the modules encoded by one of them are decomposable, with decompositions of the form
$$\xymatrixrowsep{0.25pc} \xymatrixcolsep{1.5pc}
\xymatrix{
1 \edge@/_/[dd]_{\alpha_1} \edge@/^/[dd]^{\alpha_2} &&1 \edge@/_/[dd]_{\alpha_1} \edge@/^/[dd]^{\alpha_2}  \\
&\bigoplus  \\
2 &&2
}$$
Thus, the Kac decomposition of $\bd$ relative to this component is $\bd = (1,1)+(1,1)$. The modules in the other components of $\modlad$ are generically indecomposable, and so  the Kac decomposition of $\bd$ relative to those is $\bd = (2,2)$.
\end{example}

\section{In a nutshell: Results to date}
\label{sec3}

\subsection*{\ref{sec3}.A.  Hereditary algebras, $\la = KQ$ }

As we already pointed out, Problem (1) of the program in Section \ref{sec1} is void in this case. 
Regarding Problem (2):  The deepest results regarding generic properties of the modules in $\Rep_\bd(KQ)$ pertain to the canonical decomposition of $\bd$ (see Section \ref{sec2}.C for the definition and 5 for detail).
In \cite{KacII}, Kac described the dimension vectors $\bd$ leading to generically indecomposable $\bd$-dimensional representations in terms of their generic  endomorphism rings (providing a criterion checkable from $Q$ and $\bd$).  Concerning the situation of generically decomposable $\bd$-dimensional representations, he characterized the generic dimension vectors of the corresponding direct summands via vanishing of mutual $\Ext$-spaces.  However, his description fell slightly short of providing algorithmic access to canonical decompositions, the crux lying in the $\Ext$-conditions.   This gap was filled by Schofield ten years later in \cite{Scho}.  The algorithmic nature of Kac's result became apparent by dint of another cache of generic invariants of the $KQ$-modules with fixed dimension vector. Namely, the full collection of dimension vectors generically attained on  submodules of the modules in $\Rep_\bd(KQ)$ may be computed from $Q$.    

Clearly, the dimension vectors $\bd$ of $Q$ are subject to the following dichotomy:  Either the variety $\Rep_\bd(KQ)$ contains infinitely many $\GL(\bd)$-orbits of maximal dimension, or else it contains a dense orbit; the latter situation is clearly tantamount to vanishing of the number $\mu(\bd) = \mu(\Rep_\bd(KQ))$ of generic parameters (cf.~Section \ref{sec1}).  The problem of deciding between the alternatives for given $\bd$ was in turn resolved by Kac (see \cite[Proposition 4]{KacII}).  In fact, he determined the number $\mu(\bd)$ in terms of the canonical decomposition of $\bd$.

As will become clear in Section \ref{sec4}.B, in either case, there is a single $\bd$-dimensional representation $G$ of $Q$  --  a generic module for $\Rep_\bd(KQ)$, singled out by a strong uniqueness property  --  which displays all of the essential  generic properties of the representations in $\Rep_\bd(KQ)$.  A minimal projective presentation of such a telltale module $G$ is available from $Q$ without much computational effort (see \ref{thm4.6} and \ref{prop5.4} below).  In Section \ref{sec4}, this phenomenon will be explained in the context of a general algebra $\la = KQ/I$, and then picked up again in the ensuing discussions of special cases.

Detail will follow in Section \ref{sec5}.

\subsection*{\ref{sec3}.B. Tame non-hereditary algebras }

For several classes of tame algebras $\la$, the component problem has been completely resolved.  In all of these instances, the classification of the indecomposable objects in $\lamod$ had already been achieved beforehand; it served as a pivotal tool in pinning down the components of the parametrizing varieties $\modlad$.  

Already ahead of Kac's initiative, Donald-Flanigan \cite{DoFl} and Morrison \cite{Mor} had listed the irreducible components of the Gelfand-Ponomarev algebras with $J^2 = 0$.   In \cite{Schro}, Schr\"oer classified the irreducible components of the parametrizing varieties for arbitrary Gelfand-Ponomarev algebras, that is, for the algebras $K[x,y]/\langle x^r,y^s,xy\rangle$, $r,s\ge2$.  These algebras gained prominence through work of Gelfand and Ponomarev in \cite{GePo}, where the finite dimensional representation theory of this class of tame algebras was related to the Harish-Chandra representations of the Lorentz group.  Algebras giving rise to similar module structures, in turn amenable to the methods developed by Gelfand and Ponomarev, then surfaced in the representation theory of finite groups in characteristic $2$ (see, e.g., \cite{Erd}), leading to an encompassing class of algebras, dubbed \emph{special biserial};  the name is due to the structure of the corresponding indecomposable left/right projective modules: namely the radicals of these modules are sums of two uniserials whose intersection is either zero or simple.

The component problem remains open for arbitrary special biserial algebras, but has been resolved for another subclass by Carroll and Weyman in \cite{CW}, namely for acyclic gentle string algebras.  

Moreover, a novel approach was taken by Geiss and Schr\"oer \cite{GeiSchI, GeiSchII} (as well as by Marsh and Reineke [unpublished]) towards understanding the irreducible components of the module varieties of preprojective algebras $P(Q)$, where $Q$ is a quiver of Dynkin type.  The algebras $P(Q)$ of tame, but infinite, representation type were tackled via a detour through tubular algebras.

More detail can be found in Section \ref{sec6}. 

Bounds on the number of components for certain tame algebras $\la$ may be obtained from an interesting stratification of the varieties $\modlad$ due to Richmond \cite{Rich}. Barot and Schr\"oer further explored these stratifications over canonical algebras in  \cite{BaSch}.

\subsection*{\ref{sec3}.C.  Wild non-hereditary algebras }

A solution to the problem of classifying the components of $\modlad$ by means of (computationally accessible) representation-theoretic invariants of their modules has recently been completed for \emph{truncated path algebras}, i.e., for algebras $\la$ of the form
$$KQ / \langle \text{all paths of length}\ L+1 \rangle,$$
where $L$ is a positive integer. The full picture was compiled in a sequence of installments, with contributions by Babson, Thomas, Bleher, Chinburg, Shipman, and the authors (\cite{BHT}, \cite{BCH}, \cite{irredcompI}, \cite{irredcompII}, \cite{GHZ}).

 Observe that all algebras with $J^2 = 0$ are truncated path algebras, as are all hereditary algebras.  Moreover, given any basic finite dimensional algebra $\Delta = KQ/I$, there clearly exists a unique (up to isomorphism) truncated path algebra $\Delta_{\text{trunc}}$ sharing quiver and Loewy length with $\Delta$, such that $\Delta$ is a factor algebra of $\Delta_{\text{trunc}}$.   For any dimension vector $\bd$ of $Q$, one thus retrieves $\Rep_\bd(\Delta)$ as a closed subvariety of $\Rep_{\bd} (\Delta_{\text{trunc}})$.  As we will see in Section \ref{sec9}, this embedding  provides some mileage towards the exploration of the irreducible components of $\Rep_\bd(\Delta)$ for general $\Delta$.

The pivotal asset of a truncated path algebra $\la$ lies in the fact that, among the subvarieties $\laySS$ of $\modlad$ (see Definition \ref{def2.2}), the nonempty ones are always irreducible (this follows from \cite[Theorem 5.3]{BHT}).
Since we identify isomorphic semisimple modules, we thus obtain a finite partition
$$\modlad = \bigsqcup_{\SS \in \Seq(\bd)} \laySS$$
into irreducible locally closed subvarieties.  Crucial in the present context: Since radical layerings are generically constant,
this guarantees that the irreducible components of $\modlad$ are among the closures $\overline{\laySS}$ of the subvarieties $\laySS$ of $\modlad$.  In other words, the component problem has been converted into a (significantly easier) sorting problem calling for a partition of $\Seq(\bd)$ into two camps:  The semisimple sequences $\SS$ for which $\overline{\laySS}$ is \emph{maximal} irreducible on one hand, and those $\SS$ for which $\overline{\laySS}$ is \emph{embedded} in a strictly larger $\overline{\laySS'}$ on the other. We point out that, outside the case $J^2 = 0$, the varieties $\laySS$ do not constitute a stratification of $\modlad$ in the strict sense.  Indeed, the situations where the closures $\overline{\laySS}$ are unions of $\Rep \SS'\,$s are comparatively rare; in general, the set of overlaps of the closures  of the $\laySS$ is intricate.  Feeding into the algorithmic side of the problem:  In the truncated case, it is particularly straightforward to recognize the \emph{realizable} semisimple sequences $\SS$, that is, those for which $\laySS \ne \varnothing$.   

Section \ref{sec8}, devoted to truncated path algebras, is divided into several subsections which reflect increasing degrees of effort required to arrive at a full list of the components of $\modlad$ from $Q$, $I$ and $\bd$, and to algorithmically access a large spectrum of generic properties of their modules.  The most exhaustively understood case is $J^2 = 0$.  Under this hypothesis, the solutions to Problems (1) and (2) of Section \ref{sec1} are, in fact, slightly more complete than in the hereditary case.  This is due to a geometric bridge linking ``projective incarnations" of certain comparatively small subvarieties of the $\modlad$ to analogous projective parametrizing varieties over a stably equivalent hereditary algebra.  The transfer of information will be elaborated in Section \ref{sec8}.B.

If $\la$ is \emph{local} truncated, those radical layerings $\SS = (\SS_0, \dots, \SS_L)$ with dimension vector $\bd$ which are generic  for irreducible components of $\modlad$ can be sifted out of the full set $\Seq(\bd)$ by mere inspection of the dimension vectors of the semisimple entries $\SS_l$ (Theorem \ref{thm8.8} in Section \ref{sec8}.C).  The next-simplest subcase is that of an acyclic underlying quiver $Q$.  Just as in the local case, the pivotal upper semicontinuous  map on $\modlad$, namely
$$\Theta: \modlad \longrightarrow \Seq(\bd) \times \Seq(\bd), \ \ \ x \longmapsto (\SS(M_x), \SS^*(M_x))$$
(see Section \ref{sec2}.A), detects and separates the irreducible components of $\modlad$ in this scenario (Theorem \ref{thm8.10}).
But, in contrast to the local case, we do not know of a simplified criterion that would permit dodging size comparisons among the pairs in $\Img(\Theta)$, in the process of listing the irreducible components from $Q$, $L$, and $\bd$.  On the other hand, the task is facilitated by the following facts:  Due to upper semicontinuity of $x \mapsto \SS^*(M_x)$, the minimal pairs in $\Img(\Theta)$ are of the form $(\SS, \SS^*)$ where $\SS^*$ is the generic socle layering of the modules in $\laySS$.  This generic socle layering may be recursively obtained from $\SS$, according to the formula of Theorem \ref{thm8.3} (which holds for arbitrary truncated path algebras).

For general truncated $\la$, the map $\Theta$ is known to have blind spots relative to the irreducible components of the module varieties.  It is a novel upper semicontinuous module invariant $\Gamma: \modlad \rightarrow \NN$ that compensates for this deficiency.  The generic values of  this map $\Gamma$ are in turn computationally accessible from $Q$ and the Loewy length of $\la$, but the algorithm yielding a full list of components via the $\Gamma$-test is significantly more labor-intensive than the methods we proposed for the preceding special cases.  (See \ref{sec8}.D.) 

Techniques to understand the components of the module varieties over more general algebras are still lacunary. Section \ref{sec9} contains a discussion of ways in which some of the techniques developed for truncated path algebras may be adapted to aid in identifying irreducible components in the general case.

\section{General facts about components and generic properties of their modules}
\label{sec4}

\subsection*{\ref{sec4}.A. Canonical decompositions of the irreducible components of $\modlad$ }

The results of this subsection are due to Crawley-Boevey and Schr\"oer \cite{CBS}, as is the convenient notation which will be used to convey them.
Suppose $\bd = \sum_{1 \le r \le s} \bd^{(r)}$. Given irreducible $\GL(\bd^{(r)})$-stable subvarieties $\C_r$ of $\Rep_{\bd^{(r)}} (\la)$, respectively, we denote by $\C_1 \oplus \cdots \oplus \C_s$ the $\GL(\bd)$-stable hull of the image of $\C_1 \times \cdots \times \C_s$ under the obvious map $\prod_{1 \le r \le s} \Rep_{\bd^{(r)}} (\la) \rightarrow \modlad$.  This irreducible variety is called the \emph{direct sum of}  $\C_1, \dots, \C_s$; it consists of those points $x \in \modlad$ for which $M_x \cong \bigoplus_{1 \le r \le s} M_x^{(r)}$ with $M_x^{(r)}$ in $\C_r$.  Even when the $\C_r$ are closed in the $\Rep_{\bd^{(r)}} (\la)$, the Zariski-closure of the direct sum is typically substantially larger than $\C_1 \oplus \cdots \oplus \C_s$ (see Example \ref{ex2.8}).  

The first result amounts to a Krull-Remak-Schmidt theorem for irreducible components.  
An irreducible component of some $\Rep_{\bd'}(\la)$ is called \emph{indecomposable} in case, generically, its modules are indecomposable.

\begin{theorem} \label{thm4.1} {\rm\cite[Theorem 1.1]{CBS}} Let $\C$ be an irreducible component of $\modlad$. Then there is a sum decomposition $\bd = \sum_{1\le r\le s} \bd^{(r)}$, together with indecomposable irreducible components $\C_r$ of $\Rep_{\bd^{(r)}}(\la)$, respectively, such that  $\C = \overline{\C_1\oplus \cdots \oplus \C_s}$. 
The $\bd^{(r)}$ and $\C_r$ with these properties are unique up to order.

The equation $\C = \overline{\C_1\oplus \cdots \oplus \C_s}$ is referred to as the {\rm canonical decomposition of $\C$}.
\end{theorem}

On the other hand, closures of direct sums of irreducible components need not be maximal irreducible in the ambient module variety.  Take $s = 2$, for instance, and let $\bd^{(1)} = \mathbf{e}_1$, $\bd^{(2)} =  \mathbf{e}_2$ be unit vectors.  Clearly, $\overline{\Rep_{\bd^{(1)}}(\la) \oplus \Rep_{\bd^{(2)}}(\la)}$ is maximal irreducible in $\modlad$ if and only if $\Ext_\la^1(S_1, S_2) = \Ext_\la^1(S_2, S_1) = 0$.  The following theorem furnishes the general pattern behind this trivial example.

\begin{theorem} \label{thm4.2} {\rm\cite[Theorem 1.2]{CBS}} Suppose that $\bd = \sum_{1\le r\le s} \bd^{(r)}$ and that $\C_r$ is an irreducible component of $\Rep_{\bd^{(r)}} (\la)$ for $1\le r\le s$. Then $\overline{\C_1 \oplus \cdots \oplus \C_s}$ is an irreducible component of $\modlad$ if and only if $\ext(\C_i,\C_j) = 0$ for all $i\ne j$.
\end{theorem}

This pair of results extends Kac's findings in the hereditary case (see Theorem \ref{thm5.1} below) as far as is possible in full generality, painting a clear picture of the interactions among the irreducible components of the parametrizing varieties. In favorable situations, these results should permit us to hierarchically organize these components  --  assuming they are all known  -- in terms of ``$\,\C \preceq \D \iff \D = \overline{\C \oplus \E}$ for some $\E\,$".  To concretely establish such a hierarchy for a given algebra $\la$, one would also need an algorithmic test for the vanishing of $\ext(\C_i,\C_j)$, on the model of the hereditary scenario.  An ``algebra-specific" understanding of the indecomposable pieces of the component puzzle, as well as of the modalities of gluing them together to larger components, is thus required for the purpose.

\subsection*{\ref{sec4}.B. Where to look for generic properties: Generic modules for the  components}

Many of the results from the literature referred to in this subsection are couched and proved by way of projective parametrizing varieties. For translations into the affine scenario, we refer to Section \ref{sec7}.

Suppose $\la = KQ/I$, without any restrictions on the admissible ideal $I$. Roughly, the purpose of this section is to outline the following: For each irreducible component $\C$ of any $\modlad$, there exists a $\la$-module $G$ in $\C$ such that $G$ has all essential generic properties of the modules in $\C$.  Next to securing existence, one ascertains that such a ``generic module" $G$ for $\C$ is unique, up to a special type of Morita self-equivalence of $\lamod$.  As for concrete realizations:  A minimal projective presentation of $G$ may be computed from $Q$ and a set of generators for $I$ by means of a fairly simple algorithm; the computational side will not be elaborated here.  However, for the algebras we will discuss in detail, for truncated path algebras in particular, explicit presentations of the generic modules may be simply read off the quiver (see Theorem \ref{thm4.6}).  More detail can be found in \cite[Section 4]{BHT} and \cite{hier}.

\begin{step1*} A first indication of the significance of skeleta to the component problem can be found in \ref{defobs4.4}. Intuitively, skeleta are $K$-bases of modules, made up of ``paths", which are closed under ``initial subpaths" and thus may be graphically represented by forests.  

Let $\la_0 = \la_{\text{trunc}}$ be the truncated path algebra associated with $\la$ in the sense of Section \ref{sec3}.C.  Again $L+1$ is an upper bound for the Loewy length of $\la$ and hence for that of $\la_0$. Given a semisimple $T$ in $\lamod$, let $P_0 = \bigoplus_{1\le r\le t} \la_0 \bz_r$ be a $\la_0$-projective cover of $T$ with a full sequence $(\bz_r)_{r\le t}$ of top elements; note that the semisimple objects in $\lamod$ coincide with those in $\la_0$-mod. Given that path lengths in $\la_0$ are well defined, the same is true for the lengths of the following \emph{paths in $P_0$}: these are the \emph{nonzero}  elements of the form $p\, \bz_r$, where $p$ is a path of length $\le L$ in $\la_0$. Clearly, the set of all paths in $P_0$ is a basis for $P_0$, which respects the radical layering, in the sense that the paths of length $l$ induce a $K$-basis for $J^lP_0/J^{l+1}P_0$. Such ``layer-faithful" bases are available for arbitrary $\la$-modules with top $T$, as follows.
\end{step1*}

\begin{definition} \label{def4.3} Let $\SS = (\SS_0,\dots,\SS_L)$ be a semisimple sequence of $\la$-modules with $\SS_0 = T$, and suppose that $\udim \SS = \bd$. An \emph{$($abstract\/$)$ skeleton in $P_0$ with layering $\SS$} (and dimension vector $\bd$) is any set $\S$ of paths in $P_0$ with the following properties: 
\begin{itemize}
\item  For each $l \in \{0,\dots,L\}$ and each $i \in \{1,\dots,n\}$, the number of those paths of length $l$ in $\S$ which end in the vertex $e_i$ is $\dim e_i\SS_l$; 
\item $\S$ is closed under initial subpaths, i.e., $\bigl( p\, \bz_r \in \S \  \text{and} \  p = p_2p_1  \implies p_1\, \bz_r \in \S \bigr)$. 
\end{itemize}

Moreover, given a $\la$-module $M$, we call an abstract skeleton $\S$ with layering $\SS(M)$ a \emph{skeleton of $M$} in case there exists a full sequence of top elements $z_1,\dots,z_t$ of $M$ such that the set
$$\{ p\, z_r \mid p\, \bz_r \in \S \}$$
\noindent is a $K$-basis for $M$. Note that the collection of those $p\, z_r$ for which $\len(p) = l$ then induces a $K$-basis for $J^lM/J^{l+1}M$.
\end{definition}

Due to the second condition imposed on skeleta, the $\bd$-dimensional skeleta in $P_0$  are in $1$-$1$ correspondence with forests (i.e., finite unions of tree graphs) with $|\bd|$ vertices, each vertex tagged by a primitive idempotent, such that precisely $d_i$ vertices are labeled by $e_i$ for each $i$.  We refer to Example \ref{ex2.9} to illustrate the concept.
Any module with a graph as shown on the left below has three distinct skeleta, each of them a single tree.   
$$\xymatrixrowsep{1.75pc} \xymatrixcolsep{0.75pc}
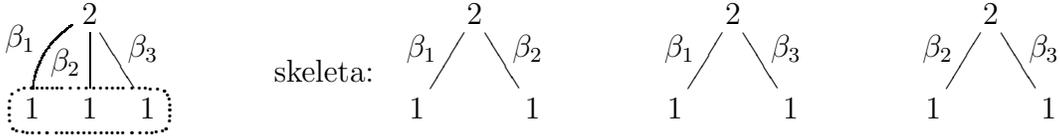
\xymatrix{
&2 \edge@/_0.6pc/[dl]_{\beta_1} \edge[d]_{\beta_2} \edge[dr]^{\beta_3} & &&&&\dropdown{\txt{skeleta:}} && &2 \edge[dl]_{\beta_1} \edge[dr]^{\beta_2} & &&& &2 \edge[dl]_{\beta_1} \edge[dr]^{\beta_3} & &&& &2 \edge[dl]_{\beta_2} \edge[dr]^{\beta_3}  \\
1 \horizpool2 &1 &1 &&&&&& 1 &&1 &&&1 &&1 &&&1 &&1
}$$
Still in the context of Example \ref{ex2.9}, let $T$ be the $\la$-module $S_1^2 \oplus S_2$, and $P_0 = \bigoplus_{1 \le r \le 3} \la \bz_r$ the distinguished $\la_0$-projective cover of $T$, where $\bz_1, \bz_2$ are normed by $e_1$ and $\bz_3$ by $e_2$.  Then the $\la$-projective cover $P$ of $T$ has a skeleton consisting of three trees, two of which are equal to the tree depicted under $\bz_1$ below, the third as depicted under $\bz_3$.  
$$\xymatrixrowsep{1.6pc} \xymatrixcolsep{1.25pc}
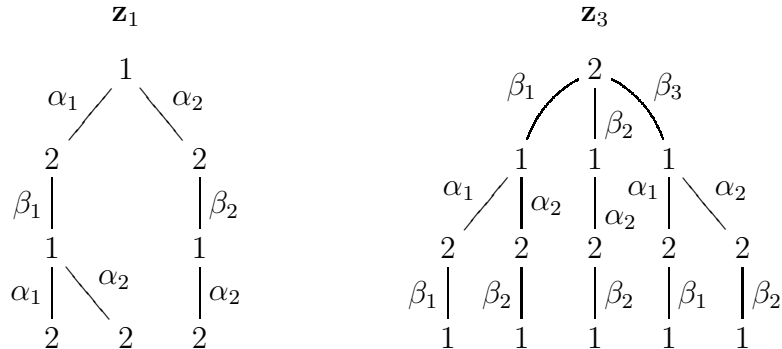
\xymatrix{
&1 \dropup{\bz_1} \edge[dl]_{\alpha_1} \edge[dr]^{\alpha_2} && &&& &&2 \dropup{\bz_3} \edge@/_/[dl]_{\beta_1} \edge[d]^(0.6){\beta_2} \edge@/^/[dr]^{\beta_3}  \\
2 \edge[d]_{\beta_1} &&2 \edge[d]^{\beta_2} & &&& &1 \edge[dl]_{\alpha_1} \edge[d]^{\alpha_2} &1 \edge[d]^(0.66){\alpha_2} &1 \edge[d]_(0.33){\alpha_1} \edge[dr]^{\alpha_2} \\
1 \edge[d]_{\alpha_1} \edge[dr]^{\alpha_2}  &&1 \edge[d]^{\alpha_2} & &&& 2 \edge[d]_{\beta_1} &2 \edge[d]_{\beta_2} &2 \edge[d]^{\beta_2} &2 \edge[d]^{\beta_1} &2 \edge[d]^{\beta_2} \\
2 &2 &2 & &&&1 &1 &1 &1 &1
}$$
\noindent Moreover, each of the modules $M$ graphed at the end of Section \ref{sec1} has precisely $4$ distinct skeleta, two of them being the forests
$$\xymatrixrowsep{1.5pc} \xymatrixcolsep{1pc}
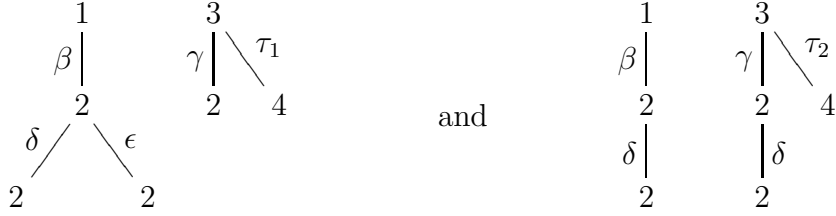
\xymatrix{
&1 \edge[d]_{\beta} & &3 \edge[d]_{\gamma} \edge[dr]^{\tau_1} & &&&&&&1 \edge[d]_{\beta} &&3 \edge[d]_{\gamma} \edge[dr]^{\tau_2}  \\
&2 \edge[dl]_{\delta} \edge[dr]^{\epsilon} & &2 &4 &&&\txt{and} &&&2 \edge[d]_{\delta} &&2 \edge[d]^{\delta} &4  \\
2 &&2 & & &&&&&&2 &&2 
}$$
\noindent Further examples can be found in Sections \ref{sec8}--\ref{sec9} below and in \cite{BHT, GHZ}.  

It is readily checked that every $\la$-module has at least one skeleton, but only finitely many. Moreover, the set of all skeleta of $M$ is generically constant, as $M$ traces the modules in any irreducible component of $\modlad$. (This is a consequence of  openness of the subvarieties $\layS \subseteq \laySS$, introduced in Observation \ref{defobs4.4} below, combined with the fact that $\SS(M_x)$ is a generic invariant, the latter meaning that each irreducible component $\C$ of $\modlad$ intersects some $\laySS$ in a dense open subset of $\C$.)  From $Q$, $I$, and a minimal projective presentation of a $\la$-module $M$, one can algorithmically test whether an abstract skeleton $\S$ is a skeleton of $M$.  The decision whether $\layS \ne \varnothing$ is algorithmic as well; see \cite[Observation 3.2]{hier}. 

\begin{defobs} \label{defobs4.4} Let $\S$ be a skeleton with layering $\SS$ and $\udim \S = \udim \SS = \bd$. The subset 
$$\layS := \{ x \in \modlad \mid \S \ \text{is a skeleton of} \ M_x \}$$
is an open subvariety of $\laySS$ (but not open in $\modlad$, in general).  See \cite[Lemma 3.8]{classifying}. In particular, each irreducible component of $\layS$ closes off to an irreducible component of $\overline{\laySS}$.

Thus, the set of irreducible components of $\modlad$ is contained in the set 
$$\{ \overline{\D} \mid \D \ \text{is an irreducible component of some} \ \layS \ \text{with}\ \udim \S =  \bd \}.$$
Consequently, the goal set at the beginning of the subsection will be met if we can secure a generic module for each irreducible component of any $\layS$.
\end{defobs}

For background and further explanation regarding Observation \ref{defobs4.4}, we refer to \ref{sec9}.A.

\begin{step2*}  (Sketch.) Let $K_0$ be the smallest subfield of $K$ with the property that $\la$ is defined over $K_0$; the latter condition means that $I$ can be generated by relations in $K_0Q$. For the moment, we assume that $K$ has infinite transcendence degree over $K_0$. Imposing this condition is innocuous:  As is explained in Observation 2.2 of \cite{irredcompII}, neither the list of irreducible components of $\modlad$ nor the corresponding collections of their essential generic properties are affected by passage from $K$ to an appropriately enlarged algebraically closed base field.

Let $\overline{K_0}$ be the algebraic closure of $K_0$ within $K$.  Evidently, every automorphism in $\Gal(K:\overline{K_0})$ gives rise to a $\overline{K_0}$-algebra automorphism of $\la$ via a twist of scalars. One checks that the corresponding twisted version of 
$\la$ is Morita equivalent to $\la$. A Morita self-equivalence of $\lamod$ is said to be \emph{$\Gal(K:\overline{K_0})$-induced} if it arises from a twist of $\la$ relative to some automorphism in $\Gal(K:\overline{K_0})$. 
\end{step2*}

\begin{defthm} \label{defthm4.5} {\rm \cite[Section 4]{BHT}} Suppose $\D$ is an irreducible component of some $\layS$, where $\S$ is a skeleton with layering $\SS$.
We call a module $G$ in $\D$ {\rm generic for} $\D$ $($or generic for $\overline{\D}$$)$ if $G$ has all generic properties of $\D$ which are invariant under $\Gal(K:\overline{K_0})$-induced Morita self-equivalences of $\lamod$. 

For every irreducible component of $\layS$, there exists a generic module. Any two ge\-ner\-ic modules for the same irreducible component of $\layS$ differ only up to a $\Gal(K:\overline{K_0})$-in\-duced Morita self-equivalence of $\lamod$. 
\end{defthm}

We refer to \cite[Supplement 1 to Theorem 4.3]{BHT} for the construction of the modules $G$ guaranteed by Theorem \ref{defthm4.5}, but will be explicit for the algebras that will be particularly relevant in Sections \ref{sec5},\ref{sec8},\ref{sec9}, namely for truncated path algebras.

\subsection*{A special case: Generic modules over a truncated path algebra $\la$}  In this situation, $K_0$ is the prime field of $K$ and $\la_0 = \la$. In preparation for Section \ref{sec8}, we remind the reader of the fact that, over a truncated path algebra $\la$, all of the varieties $\laySS$ are irreducible. In particular, $\layS$ is open dense in $\laySS$ whenever $\S$ is a skeleton with layering $\SS$. Indeed, $\layS \ne \varnothing$ is automatic in the present situation (i.e., $\SS$ is realizable if and only if there exists a skeleton with this layering; see also Criterion \ref{crit8.2} below).  The skeleta with layering $\SS$  may be directly read off the quiver $Q$.

Consequently, generic modules for the nonempty varieties $\laySS$ are also available at a glance from $Q$ as follows.  Let $\S \subseteq P_0$ be any skeleton with layering $\SS$; here $P_0$ is a $\la$-projective cover of $\SS_0$, say $P_0 = \bigoplus_{1\le r\le t} \la \bz_r$ for some top elements $\bz_r$, as in \ref{sec4}.A.  A path $q\, \bz_r$ in $P_0$ is called \emph{$\S$-critical} if it does not belong to $\S$, but factors in the form $q = \alpha \cdot q_1 \, \bz_r$ where $\alpha$ is an arrow and $q_1\, \bz_r$ belongs to $\S$. Clearly, the $\S$-critical paths may in turn be listed by mere inspection of $Q$.  Our presentation of a generic module $G = P_0/ \Omega^1(G)$ for $\laySS$ is in terms of expansions of the $\S$-critical paths along a basis for $G$ induced from the linearly independent subset $\S$ of $P_0$.

\begin{theorem} \label{thm4.6} {\rm\cite[Theorem 5.12]{BHT}} Let $\SS$ and $P_0$ be as above.  Given any skeleton $\S$ with layering $\SS$, the following module $G$ is  generic for $\overline{\laySS}$:   
$$G = P_0 / R(\S), \ \ \text{where} \ \ 
R(\S) = \sum_{\substack{q\, \bz_r \\ \S\text{\rm-critical}}} \la \biggl( q\, \bz_r \ \ - \sum_{\substack{p\, \bz_s \in \S, \ \term(p)=\term(q), \\ \len(p\, \bz_s) \ge \len(q\, \bz_r)}} x_{q\, \bz_r, \, p\, \bz_s} p\, \bz_s \biggr);$$
here the $x_{-, -}$ are scalars in $K$ which are algebraically independent over $K_0$. 
\end{theorem}

In general, the cardinality of the $K_0$-algebraically independent set of scalars $x_{-, -}$ will be significantly larger than the generic number $\mu(\laySS) = \dim \laySS - \dim \text{orbit}(G)$ of parameters for $\laySS$; indeed, examples abound where the number of parameters in the above presentation of $G$ is redundant.

\section{More detail on the hereditary case}
\label{sec5}

Kac provided the following characterization of the canonical decomposition of a dimension  vector $\bd$ of $Q$.  For brevity of formulation, we use Schofield's notational convention:
\begin{align*}
\ext(\bd, \bd') &= \ext(\Rep_\bd(KQ), \Rep_{\bd'}(KQ))  \\
&= \min \{\dim \Ext^1_\la(M, M') \mid \udim M = \bd, \ \udim M' = \bd'\}.
\end{align*} 

\begin{theorem} \label{thm5.1} {\rm\cite[Proposition 3]{KacII}} A decomposition of the dimension vector $\bd$, say $\bd = \sum_{1 \le r \le s} \bd^{(r)}$, is the canonical decomposition $($= Kac decomposition$)$ of $\bd$ if and only if the following two conditions are satisfied:
\begin{itemize}
\item For  $1 \le r \le s$, the representations of $KQ$ with dimension vector $\bd^{(r)}$ generically have endomorphism rings equal to $K$. $($Such dimension vectors are called  {\rm Schur roots} of $Q$$)$.
\item $\ext(\bd^{(r)}, \bd^{(u)}) = 0$ whenever $r, u \in \{1, \dots, s\}$ are distinct.
\end{itemize}
\end{theorem}

Schofield filled the remaining gap between Kac's theoretical description of canonical decompositions and their algorithmic accessibility by analyzing further generic invariants of the $\bd$-dimensional representations, namely the dimension vectors which are generically attained on submodules of the $\bd$-dimensional modules.

\begin{theorem} \label{thm5.2} {\rm\cite[Theorem 3.3]{Scho}} For dimension vectors $\bd$ and $\bd'$ of $Q$, the following are equivalent:

{\bf(a)} Generically, the representations of $KQ$ with dimension vector $\bd$ have a subrepresentation with dimension vector $\bd'$.

{\bf(b)} Every representation of $KQ$ with dimension vector $\bd$ has a subrepresentation with dimension vector $\bd'$.

{\bf(c)} $\ext(\bd', \bd - \bd') = 0$.
\end{theorem}

Asking that the $KQ$-modules with dimension vector $\bd$ generically have submodules with dimension vector $\bd'$, as well as submodules with dimension vector $\bd - \bd'$, is thus equivalent to imposing the equalities $\ext(\bd', \bd - \bd') = \ext(\bd - \bd', \bd') = 0$.  In light of the fact that $\dim \Ext^1(-, - )$ is upper semicontinuous, these $\ext$-values are attained on a dense open subset of $\Rep_{\bd'}(KQ) \times \Rep_{\bd - \bd'}(KQ)$, whence we recoup the second part of Kac's result, Theorem \ref{thm5.1}.  In particular, we conclude that the canonical decomposition of $\bd$ may be gleaned from the  set $\operatorname{Sub}(\bd)$ of dimension vectors which are generically attained on the submodule lattices of the modules in $\Rep_{\bd}(KQ)$.  Schofield further recoined condition (c) of Theorem  \ref{thm5.2} into a format (involving the Euler form of $Q$) permitting to recursively reduce the vanishing test for $\ext( - , -)$ to successively smaller dimension vectors \cite[Theorem 5.4]{Scho}; this allows for computation of $\operatorname{Sub}(\bd)$ from $Q$.  In this connection, we point to a simplification of Schofield's algorithm due to Derksen and Weyman \cite[Section 4]{DeWe}.

\begin{return2.9*}  The set of dimension vectors generically attained on the submodules of the modules in $\Rep_{(2,2)}(KQ)$ is 
$$\operatorname{Sub}(2,2) = \{ (0,0), (2,2), (1,1), (1, 2), (0,1), (0,2)\},$$
which confirms the canonical decomposition of $\bd = (1,1) + (1,1)$.  In particular,  the only dimension vectors $\le \bd$ excluded from $\operatorname{Sub}(2,2)$ are $(1,0)$, $(2, 0)$, and $(2,1)$. 
\end{return2.9*} 

Clearly, the postulate that the generic modules for $\Rep_\bd(KQ)$ should all belong to the same isomorphism class is tantamount to the existence of a dense orbit in $\Rep_\bd(KQ)$.  Once the canonical decomposition of $\bd$ is available, the issue may be decided by means of the following result of Kac.

\begin{theorem} \label{thm5.3} {\rm\cite[Proposition 4]{KacII}} If $\bd = \sum_{1 \le r \le s} \bd^{(r)}$ is the canonical decomposition of $\bd$, then the generic number of parameters of $\Rep_\bd(KQ)$ 
is 
$$\mu(\bd) = \sum_{1 \le r \le s} \bigl(1 - \langle \bd^{(r)}, \bd^{(r)} \rangle \bigr),$$
where $\langle -,- \rangle$ denotes the Euler form of $Q$.  In particular, $\Rep_\bd(KQ)$ contains a dense $\GL(\bd)$-orbit precisely when $\sum_{1 \le r \le s} \bigl(1 - \langle \bd^{(r)}, \bd^{(r)} \rangle \bigr) = 0$.
\end{theorem}

An explicit presentation of ``the" generic $\bd$-dimensional $KQ$-module $G = G(\bd)$ is available from Theorem \ref{thm4.6}. (It suffices to observe that $KQ$ is a truncated path algebra.) It is based on the generic radical layering of the modules in $\Rep_\bd(KQ)$, which is supplied by the following recursion formula.  Here $\mathbf{A}$ denotes the adjacency matrix of $Q$, i.e., $\mathbf{A}$ is the $|Q_0| \times |Q_0|$-matrix whose entry $\mathbf{A}_{ij}$ counts the number of arrows from $e_i$ to $e_j$. 

\begin{proposition} \label{prop5.4} {\cite[Proposition 4.1]{irredcompII}}  Suppose that the lengths of the paths in $Q$ are bounded by $L$.  Given any dimension vector $\bd$ of $Q$, let $\SS(\bd) = (\SS_0, \dots, \SS_L)$ be the generic radical layering of the modules in $\Rep_\bd(KQ)$.  Then the dimension vectors $\mathbf{t}^{(l)} = \udim \SS_l$ for $0 \le l \le L$ are given by
$\mathbf{t}^{(0)}  = \sup\, \{\mathbf{0}, \, \bd - \bd \cdot \mathbf{A} \}$, and
$$\mathbf{t}^{(l+1)} = \sup\, \biggl\{\mathbf{0}, \, \biggl(\bigl(\bd - \sum_{i \le l} \mathbf{t}^{(i)}\bigr) - \bigl(\bd - \sum_{i \le l} \mathbf{t}^{(i)}\bigr)  \cdot \mathbf{A} \biggr)  \biggr\},$$
where the suprema are taken with respect to the componentwise partial order on $\ZZ^n$.
\end{proposition}

As for the reach of the generic theory of $\modlad$:  Clearly, the Loewy lengths of the $\bd$-dimensional $KQ$-modules are generically constant.  (The generic value is the maximal one, i.e.,  the least $m$ such that the entry $\SS_m$ of the generic radical layering $\SS = \SS(\bd)$ vanishes.)  Thus, the above generic results only reach the $\bd$-dimensional modules of maximal Loewy length.  Modules of any smaller Loewy length evidently arise as representations of suitable truncations of $KQ$; as such, they are in turn generically accessible via the results in \ref{sec8}.D.

\section{More detail on the tame non-hereditary case}
\label{sec6}

The upcoming sample results are aimed at illustration, rather than completeness.

Recall from Section \ref{sec3}.B that the Gelfand-Ponomarev algebras are those of the form $\la = KQ/ \langle \alpha^r, \beta^s, \alpha \beta, \beta \alpha \rangle$, where $r,s\ge2$ and $Q$ is the quiver
$$\xymatrix{
\bullet \ar@(ul,dl)_{\alpha} \ar@(ur,dr)^{\beta}
}$$

\subsection*{$\bullet$ {\bf The Gelfand-Ponomarev algebra with $J^2 = 0$}}  Work of Donald-Flanigan \cite{DoFl} and Morrison \cite{Mor} showed, in particular, that the only irreducible components containing infinitely many orbits of maximal dimension occur for even dimension $d = 2m$.  The generic modules for these components are of the form $\bigoplus_{1 \le i \le m} \la / \la (\beta - x_i \alpha)$, where $x_1, \dots, x_m \in K$ are algebraically independent over the prime field.  All other components are closures of single orbits, each represented by a generic module that is unique up to isomorphism.  The generic modules occurring for these latter components of $\Rep_d(\la)$ are precisely the $d$-dimensional direct sums of  modules of the form $(U_k)^u \oplus (U_{k+1})^v$, where $U_k$ is the $(2k+1)$-dimensional string module with graph
$$\xymatrixrowsep{1.5pc} \xymatrixcolsep{0.75pc}
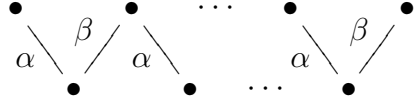
\xymatrix{
\bullet \edge[dr]_{\alpha} &&\bullet \edge[dl]_{\beta} \edge[dr]_{\alpha} &\ar@{}[r]|{\cdots} &&\bullet \edge[dr]_{\alpha} &&\bullet \edge[dl]_{\beta}  \\
&\bullet &&\bullet &\ar@{}[r]|{\cdots} &&\bullet
}$$
\noindent next to the  $K$-duals of such modules.  In reference to the results of Section \ref{sec5}, we note that $\Ext_\la^1(U_k, U_l)  = \Ext_\la^1(U_l, U_k)  = 0$ if and only if $|k-l| \le 1$ (see \cite[Proof of Theorem 5.1]{Mor}).  In Section \ref{sec8}.B, this picture will be integrated into the general solution to the component problem for algebras with vanishing radical square.   In particular, it will be seen that, for any local algebra  $\la$ with $J^2 = 0$ and $\dim J =r$, the irreducible components of $\Rep_d(\la)$ are in bijective correspondence with the partitions $d = u + v$ such that $u \le rv$ and $v \le ru$. 

Closely related to this algebra, with respect to the component problem, is the \emph{Carlson algebra} $K[x,y]/\langle x^2, y^2\rangle$. In \cite{RiRuSm} Riedtmann, Rutscho and Smal\o\ determined the irreducible components of its module varieties in terms of affine equations.

\subsection*{$\bullet$ {\bf Arbitrary Gelfand-Ponomarev algebras}} Schr\"oer's solution of the component problem for arbitrary Gelfand-Ponomarev algebras \cite{Schro} is very complete, in that it again allows to specify generic modules for the irreducible components of the varieties $\modlad$.  His classification separately describes the components with infinitely many orbits of maximal dimension and those containing dense orbits.  We include a graphic illustration of the outcome in a special case, addressed in \cite[Theorem 1.1]{Schro}:  Namely, if $d = r = s \ge 2$, then $\Rep_d(\la)$ has precisely $d - 1$ irreducible components $\C_i$, $1 \le i \le d-1$, each including infinitely many orbits of maximal dimension.  The generic modules for the $\C_i$ may be visualized as follows:
$$\xymatrixrowsep{1.25pc} \xymatrixcolsep{1pc}
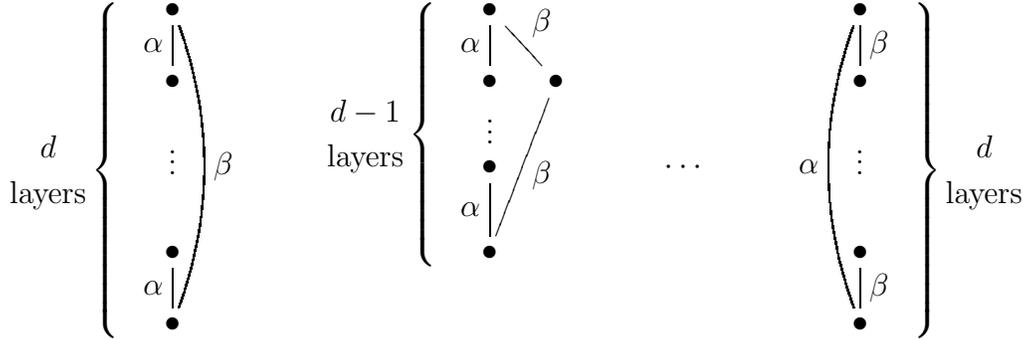
\xymatrix{
&\bullet \edge[d]_{\alpha} \edge@/^1pc/[dddd]^{\beta}  &&&&&&\bullet \edge[d]_{\alpha} \edge[dr]^{\beta}  &&&&&&\bullet \edge[d]^{\beta} \edge@/_1pc/[dddd]_{\alpha}  &  \\
\save[-1,0].[3,0]!C *\frm{\{} = "lbl1" \restore \save"lbl1"+L \drop++!R{\begin{matrix} d \\ \txt{layers}\end{matrix}} \restore &\bullet  &&&&&\save[-1,0].[2,0]!C *\frm{\{} = "lbl1" \restore \save"lbl1"+L \drop++!R{\begin{matrix} d-1 \\ \txt{layers}\end{matrix}} \restore &\bullet \ar@{}[d]|{\vdots} &\bullet \edge[ddl]^{\beta}  &&&&&\bullet &\save[-1,0].[3,0]!C *\frm{\}} = "lbl3" \restore \save"lbl3"+R \drop++!L{\begin{matrix} d \\ \txt{layers}\end{matrix}} \restore  \\
&\vdots  &&&&&&\bullet \edge[d]_{\alpha}  &&&\cdots&&&\vdots   \\
&\bullet \edge[d]_{\alpha}  &&&&&&\bullet  &&&&&&\bullet \edge[d]^{\beta}  \\
&\bullet  &&&&&&  &&&&&&\bullet  &
}$$
In each case, the open subset of $\C_i$ consisting of the orbits of maximal dimension thus has a moduli space isomorphic to $\AA^1$.  The general description of the irreducible components of the $\Rep_d(\la)$ is combinatorially too involved for inclusion here.

\subsection*{$\bullet$  {\bf Gentle algebras}} 
A \emph{gentle string algebra} is an algebra of the form $\la = KQ/I$ where $I$ is generated by certain paths of length $2$ such that $Q$ and $I$ have the following additional properties: $\bullet$ For each vertex $v$, there are at most two arrows leaving $v$ and at most two arrows entering $v$; $\bullet$  Whenever $\alpha$ is an arrow and $\beta_1$, $\beta_2$ are distinct arrows ending in $\strt(\alpha)$, precisely one of the paths $\alpha\beta_i$ belongs to $I$; $\bullet$  Whenever $\alpha$ is an arrow and $\gamma_1, \gamma_2$ are distinct arrows starting in $\term(\alpha)$, precisely one of the paths $\gamma_i \alpha$ belongs to $I$.
 Assuming $Q$ to be acyclic, Carroll and Weyman \cite{CW} determined the irreducible components of the varieties $\modlad$ in terms of certain functions $r : Q_1 \rightarrow \ZZ_{\ge0}$ called \emph{rank sequences}.  When the set of rank sequences is equipped with the componentwise partial order, the irreducible components of any $\modlad$ are the sets
$$\{ x \in \modlad \mid \rank x_\alpha \le r(\alpha) \ \ \text{for all} \ \ \alpha \in Q_1 \},$$
where $r$ traces the maximal rank sequences \cite[Proposition 5.2]{CW}. Generic modules for these components were constructed by Carroll \cite[Corollaries 3.6, 3.8]{Car}.

\subsection*{$\bullet$ {\bf Tubular algebras}}  In \cite{GeiSchI}, Geiss and Schr\"oer provided a classification of the irreducible components of $\modlad$ when $\la$ is a certain type of ``tubular extension" of a tame concealed algebra $\la_0$.  Minimal background: An algebra is \emph{tame concealed} if it results from a tame hereditary algebra via tilting by a preprojective tilting module, that is, by a tilting module $T \in \lamod$ with the property that $\tau^{k} (T)$ is projective for some $k \ge 0$;  here $\tau$ is the Auslander-Reiten translate.  Very roughly, a \emph{tubular extension} of $\la_0$ is an extension resulting from a finite sequence of modified one-point extensions at modules coming from distinct tubes in the Auslander-Reiten quiver of $\la_0$.  This class of tame algebras was introduced and analyzed by Ringel (see \cite{RinI}), in light of his observation that it constitutes a large class of algebras whose module categories inherit pivotal assets from those of tame hereditary algebras: Namely, the Auslander-Reiten quiver consists of a preinjective and a preprojective component, next to infinitely many $\PP^1(K)$-families of tubes.

\subsection*{$\bullet$  {\bf Canonical decompositions over preprojective algebras}}  The \emph{preprojective algebra} $\la = P(Q)$ of a quiver $Q = (Q_0, Q_1)$ is obtained as follows.  Supplement each arrow $\alpha: e_i \rightarrow e_j$ in $Q_1$ by an arrow $\alpha^*: e_j \rightarrow e_i$ to arrive at a new quiver $\overline{Q} = (Q_0, Q_1 \sqcup Q_1^*)$, where $Q_1^* = \{\alpha^* \mid \alpha \in Q_1\}$.  Then $\la = K\overline{Q} / \langle \sum_{\alpha \in Q_1} \alpha^* \alpha - \alpha \alpha^* \rangle$.  We refer to \cite{Rin} for further background.

In \cite{GeiSchI} and \cite{GeiSchII}, Geiss and Schr\"oer extended work of Marsh and Reineke [unpublished] regarding irreducible components of preprojective algebras based on simply-laced Dynkin graphs.  In this scenario, the irreducible components of the $\modlad$ are known to correspond  to the elements of a canonical basis for the negative part of the quantized enveloping algebra of the Lie algebra associated with $Q$ \cite{KaSa}. Geiss and Schr\"oer classified the irreducible components for the tame cases $Q = \AA_5$ and $Q =  \DD_4$.  Beyond that, they obtained an interesting limitation on the number of distinct summands arising in the canonical decomposition of certain components; their bound also applies to the wild preprojective algebras based on the quivers of Dynkin type $\AA_n$ for $n \ge 6$, $\DD_n$ for $n \ge 5$, and $\EE_6$, $\EE_7$, $\EE_8$.  Namely, whenever an irreducible component $\C$ of some $\modlad$ contains a dense orbit represented by a module without self-extensions, the canonical decomposition of $\C$ in the sense of Section \ref{sec4}.A is of the form $\C = \overline{\C_1^{m_1} \oplus \cdots \oplus \C_u^{m_u}}$, where $u$ is bounded from above by the number of positive roots of $Q$.

\section{Projective parametrizing varieties}
\label{sec7}

We describe alternative projective varieties designed to parametrize classes of $\bd$-di\-men\-sion\-al $\la$-modules and explain how they relate to the affine parametrizing varieties in $\modlad$ encoding the same classes of modules.  It is in these projective varieties that the proofs of the theorems of Section \ref{sec8} are anchored.  However, in the present survey, the only explicit applications of the projective parametrizing varieties occur in \ref{sec8}.B, next to brief appearances in Sections \ref{sec8}.C and \ref{sec9}.    

\subsection*{\ref{sec7}.A. The ``small" projective parametrizing varieties $\grasstbd$ and $\grassSS$}

Let $T \in \lamod$ be semisimple, and $P$ a  $\la$-projective cover of $T$.  Moreover, take $\bd$ to be a dimension vector of $\la$ with $\udim T \le \bd$, and set $d = |\bd|$.  By $\Gr(\dim P - d,\, JP)$ we denote the classical Grassmann variety of $(\dim P-  d)$-dimensional $K$-subspaces of $JP$.

\begin{param*} We define $\grasstbd$ to be the subset of $\Gr(\dim P - d, \, JP)$ consisting of those points $C$ which are $\la$-submodules of $JP$ and have the additional property that $\udim P/C = \bd$.  Note that, due to $C \subseteq JP$, the factor modules $P/C$ all have top $T$. (Recall: We identify isomorphic semisimple modules.)

Then $\grasstbd$ is a closed subvariety of $\Gr(\dim P- d, \, JP)$.  In particular, $\grasstbd$ is a projective variety.   Clearly, the map
\begin{align*}
\grasstbd &\longrightarrow  \{\text{iso classes of}\ \la\text{-modules with dim vector} \ \bd\ \text{and top}\ T\}  \\
C &\longmapsto \ \ \text{iso class of}\ P/C
\end{align*}
is surjective and thus parametrizes the $\bd$-dimensional $\la$-modules with top $T$, up to isomorphism.  Moreover,
the natural (morphic) action of the algebraic group $\autlap$ on $\grasstbd$ furnishes a partition of $\grasstbd$ into the subsets corresponding to the different isomorphism classes of the modules under consideration:  Indeed, the $\autlap$-orbits are in $1$-$1$ correspondence with these isomorphism classes.
\end{param*}

To compare this projective parametrization with the corresponding (quasi-) affine one, denote the locally closed subvariety of $\modlad$ consisting of the points $x$ with $\Top(M_x) = T$ by $\Rep^T_\bd$.  Evidently, this subvariety is stable under the $\GL(\bd)$-action of $\modlad$. 

\begin{proposition} \label{prop7.1}  \cite[Proposition C]{unifinIV}  Consider the natural inclusion-preserving and -re\-flec\-ting bijection between the $\GL(\bd)$-stable subsets of $\Rep^T_\bd$ on one hand and the $\autlap$-stable subsets of $\grasstbd$ on the other, which is defined by the requirement that it  pairs orbits encoding  isomorphic modules. This correspondence  preserves and reflects openness, closures, irreducibility, and smoothness. 
\end{proposition}

Note in particular: If $\SS$ is a $\bd$-dimensional semisimple sequence with $\SS_0 = T$, then the above correspondence pairs the locally closed subvariety  $\grassSS$ of $\grasstbd$, which consist of the points $C$ with $\SS(P/C) = \SS$, with the previously defined subvariety $\laySS$ of $\modlad$.

\subsection*{\ref{sec7}.B.  The ``big" projective parametrizing varieties $\grassbd$ and $\biggrassSS$}

Given $\bd$, fix a $\la$-projective cover $\bP$ of $\bigoplus_{1 \le i \le n} S_i^{d_i}$.  In other words, $\bP$ is minimal projective relative to the property that all $\la$-modules with dimension vector $\bd$ arise as factor modules of $\bP$.  

\begin{param2*}  
We define $\grassbd$ to be the closed subvariety of $\Gr\bigl( \dim \bP -  d,\,\bP\bigr)$ consisting of those points
$C$ which are $\la$-submodules of $\bP$ and have the additional property that $\udim \bP/C = \bd$.   

 In particular, $\grassbd$ is a projective variety, and the map
\begin{align*}
\grassbd &\longrightarrow  \{\text{iso classes of}\ \la\text{-modules with dim vector} \ \bd\}  \\
C &\longmapsto \ \ \text{iso class of}\ \bP/C
\end{align*}
is surjective.  The role played by $\autlap$ in \ref{sec7}.A is taken over by the larger automorphism group $\Aut_\la(\bP)$ in the broader scenario.
\end{param2*}

In complete analogy with Proposition \ref{prop7.1}, one obtains an inclusion-preserving/re\-flec\-ting bijection between the $\GL(\bd)$-stable subsets of $\modlad$ on one hand and the $\Aut_\la(\bP)$-stable subsets of $\grassbd$ on the other.  In turn, this bijection preserves and reflects openness, closures, irreducibility, and smoothness.  Under this broader correspondence, any subvariety of the form $\laySS$ of $\modlad$ corresponds to the locally closed subvariety $\biggrassSS$ of $\grassbd$ which consists of the points $C$ with $\SS(\bP/C) = \SS$.  Observe, in particular, that $\biggrassSS$ encodes the same isomorphism classes of modules as $\grassSS$, but is significantly larger in general. Since the modules in any irreducible component have generically constant tops, it is therefore advantageous to operate in the smaller setting of \ref{sec7}.A, ahead of final size comparisons of the closures in $\grassbd$ of the components of the various $\biggrassSS$.

\section{The wild case:  Focus on truncated path algebras}
\label{sec8}

In this section, we restrict our attention to truncated path algebras
$$\la = KQ / \langle \text{all paths of length}\ L+1 \rangle.$$
Subsequently (in Section \ref{sec9}), we will sketch and exemplify a strategy to apply information garnered in the truncated case to more general path algebras modulo relations.  In light of the discussion in \ref{sec3}.C, we are confronted with a selection problem raised by the following facts:

\begin{theorem} \label{thm8.1} {\rm\cite[Section 5]{BHT} ($\la$ truncated)} For any realizable semisimple sequence $\SS$, the variety $\laySS$ is irreducible.  Moreover, all irreducible components of $\modlad$ are among the closures $\overline{\laySS}$, where $\SS$ traces the realizable semisimple sequences with dimension vector $\bd$. 
\end{theorem}

 Thus, our task is to characterize those sequences $\SS$ for which $\laySS$ is not contained in $\overline{\laySS'}$ for any semisimple sequence $\SS' < \SS$.
In light of the duality $\lamod$ $\leftrightarrow \text{mod-}\la$, the situation is actually symmetric relative to radical and socle layerings.  The choice of placing the emphasis on radical layerings was prompted by the prior development of techniques for modules with fixed top (see Section \ref{sec7}).  

Recall that, for any dimension vector $\bd$ of $Q$, we have a map 
$$\Theta : \modlad \rightarrow \Seq(\bd) \times \Seq(\bd), \ \ x \mapsto (\SS(M_x), \SS^*(M_x)).$$
It will provide the leitmotif for sorting the sequences in $\Seq(\bd)$ according to their component status.   Upper semicontinuity of $\Theta$ (Proposition \ref{prop2.3}) places the primary focus on those $\SS$ which give rise to minimal elements in $\Img(\Theta)$; see \ref{sec2}.B. In particular, we know:  If $(\SS,\SS^*)$ is a minimal pair in $\Img(\Theta)$, then $\overline{\laySS}$ is an irreducible component of $\modlad$.

Clearly, the first entries of the pairs in $\Img(\Theta)$ are precisely the realizable $\bd$-dimensional semisimple sequences.  Due to semicontinuity, we are, moreover, only interested in those pairs $(\SS, - ) \in \Img(\Theta)$ whose second slots are occupied by the \emph{generic} socle layering of the modules in $\laySS$, respectively.  The first subsection is dedicated to making the relevant pairs in the image of $\Theta$ concretely accessible from $Q$ and $L$.

\subsection*{\ref{sec8}.A. Realizability criterion and generic socles}

\begin{criterion} \label{crit8.2} {\rm\cite[Observation 5.2]{BHT}, \cite[Criterion 3.2]{irredcompI} ($\la$ truncated)} Let $\bA$ denote the adjacency matrix of $Q$.  For any semisimple sequence $\SS = (\SS_0, \dots, \SS_L)$ in $\lamod$, the following conditions are equivalent:
\begin{itemize}
\item $\SS$ is realizable.
\item For $0\le l\le L-1$, the sequence $(\SS_l, \SS_{l+1})$ is realizable over $\la/J^2$.
\item $\udim \SS_{l+1} \le \udim \SS_l \cdot \bA$ for $0\le l\le L-1$.
\end{itemize}

\noindent{\rm The final condition permits to decide realizability of $\SS$ at a glance.  Indeed, it says that, for each $l < L$ and $k \in \{1, \dots, n\}$, the dimension $\dim e_k \SS_{l+1}$ is bounded from above by $\sum_{j=1}^n (\dim e_j \SS_l) \cdot |\{\alpha \in Q_1 \mid \strt(\alpha) = e_j \ \text{and}\ \term(\alpha) = e_k\}|$.}  
\end{criterion}

Listing the irreducible components of the variety $\modlad$ will involve comparisons of pairs in the image of $\Theta$ under the componentwise dominance order on the codomain. The execution of this task is rendered much more efficient by the facts that $\bullet$ for any realizable semisimple sequence $\SS$, the unique smallest socle layering attained on the modules in $\laySS$ is the generic one, and $\bullet$ this generic socle layering $\SS^*$ may be computed from $\SS$ by way of the theorem below. We use the notation $E_1(X) = \soc(E(X)/X)$ for any module $X$; here $E(X)$ is the injective envelope of $X$.  For any semisimple $X \in \lamod$, the module $E_1(X)$ equals $\SS_1^*(E(X))$ and is readily gleaned from the quiver $Q$; this is, in fact, dual to the considerations targeting the subfactor $\SS_1(P(X)) = J P(X)/ J^2 P(X)$ of a projective cover $P(X)$ of a semisimple module $X$. 

\begin{theorem} \label{thm8.3} {\rm \cite[Theorem 3.8]{irredcompII} ($\la$ truncated)} Denote by $\bB$ the transpose of the adjacency matrix of $Q$. Let $\SS$ be a realizable semisimple sequence, set $\SS_{L+1} = 0$, and let $\SS^* = (\SS_0^*,\dots,\SS_L^*)$ be the generic socle layering of the modules in $\laySS$.

{\bf(a)}  The generic socle $\SS_0^*$ of the modules in $\laySS$ is given by its dimension vector 
$$\udim \SS_0^* = 
\sup \biggl\{ \sum_{L - j \le l\le L} \biggl( \udim \SS_l - \udim \SS_{l+1} \cdot \bB \biggr) \biggm| 0 \le j\le L \biggr\}.$$

{\bf(b)} Generically, the quotients $M/\soc M$ for $M$ in $\laySS$ have radical layering $\SS' = (\SS'_0, \dots, \SS'_{L-1},0)$ where the  vectors $\udim \SS'_l$ are recursively given by $\udim \SS'_L = 0$ and
$$\udim \SS'_{L-m} = \inf \biggl\{ \udim \SS_{L-m}, \,  \biggl(\, \sum_{0\le j\le m-1} \udim E_1(\SS_{L-j}) \biggr) - \biggl(\, \sum_{0\le j\le m-1} \udim \SS'_{L-j} \biggr) \biggr\}$$
for $1\le m\le L$. 

The generic socle layering of the modules in $\laySS'$ is $(\SS^*_1,\dots,\SS^*_L,0)$.

{\bf(c)} The higher entries of $\SS^*$ are obtained recursively from parts {\bf (a)} and {\bf (b)}.
\end{theorem}

\subsection*{\ref{sec8}.B. The most complete generic picture: $J^2=0$}

Throughout this subsection, we assume $\la = KQ/I$ where $I$ is generated by all paths of length two. We give ample space to this case due to its level of completeness.

First we will classify the irreducible components of $\modlad$.  Then we will use a collection of geometric bridges between the representation theory of $\la$ and that of a stably equivalent hereditary algebra $\lahat$ (as announced in Section \ref{sec3}) to describe further generic properties of the modules in the components.

In the current situation, we may communicate $\Theta$ in clipped form. Namely, in light of $\SS(M) = (M/JM, JM)$ and $\SS^*(M) = (\soc M, M/\soc M)$, it suffices to record $M/JM$ and $\soc M$ to pin down the value of $M$ under $\Theta$. Hence, we may now convey $\Theta$ in the form $x \mapsto ( \Top M_x,  \soc M_x)$, the componentwise dominance order boiling down to componentwise inclusion. In other words, $(T,U) \le (T',U')$ if and only if $T \subseteq T'$ and $U \subseteq U'$.  In our present situation, $\Theta$ detects and separates all irreducible components of $\modlad$. The following result refines this information. In fact, it shows that the partition of $\modlad$ into locally closed subvarieties $\laySS$ is a stratification in this exceptional case, the boundary of any stratum being the union of the strata with larger $\Theta$-values. 

\begin{theorem} \label{thm8.4} {\rm\cite[Theorem 3.6]{BCH}}  $(J^2=0)$ Let $\bd$ be a dimension vector and $\SS$, $\SStilde$ semisimple sequences with dimension vector $\bd$. Then
$$\Rep \SStilde \subseteq \overline{\laySS} \ \ \iff \ \  \SStilde_0 \supseteq \SS_0 \ \text{and} \,\  \SStilde_0^* \supseteq \SS_0^* \, .$$
In particular, the irreducible components of $\modlad$ are precisely those closures $\overline{\laySS}$ for which the pair $(\SS_0, \SS^*_0)$ is minimal in $\Img(\Theta)$. 
\end{theorem}

For $J^2 = 0$, there is a particularly efficient computational route to the irreducible components of the module varieties; see Proposition 3.9 through Example 3.11 in \cite{BCH}. Moreover, the full collection of modules in any given irreducible component is characterized with ease:

\begin{proposition} \label{prop8.5} {\rm\cite[Corollary 4.2]{BCH}} $(J^2=0)$ For any semisimple sequence $\SS$ in $\Seq(\bd)$, the modules in $\overline{\laySS}$ are precisely those $\bd$-dimensional modules which have the form $X_1 \oplus X_2$ with $\Top X_1 = \SS_0$ and $X_2$ semisimple.
\end{proposition}

A major cache of generic information is opened up by the fact that $\la$ is stably equivalent to the following hereditary algebra $\lahat = K\Qhat$. The quiver $\Qhat$ (known as the \emph{separated quiver} of $Q$) has twice as many vertices as $Q$, namely $\Qhat_0 = \{e_1,\dots,e_n, \ehat_1,\dots,\ehat_n \}$. The arrows in $\Qhat$ are of the form $\alphahat$, where $\alpha$ traces $Q_1$ and $\alphahat$ has source $e_i$ and target $\ehat_j$ if $\alpha$ is an arrow from $e_i$ to $e_j$. Note that the hereditary algebra $\lahat$ in turn has vanishing radical square; in fact, the vertices $\ehat_1,\dots,\ehat_n$ are sinks.
Accordingly, the $2n$ simple left $\lahat$-modules may be split into two camps as follows: $\hatS(e_i) = \lahat e_i/ \Jhat e_i$ and $\hatS(\ehat_j) = \lahat \ehat_j$.  Our sequencing of the entries of the dimension vectors of $\Qhat$ follows the ordering of $\Qhat_0$ given above.

For instance, if $\la$ is local, i.e., if $Q$ consists of a single vertex with finitely many loops -- say $r$ loops -- then $\Qhat$ is the generalized Kronecker quiver with two vertices and $r$ equidirected arrows. A particularly complete generic picture of the $\lahat$-modules in this situation can be found in \cite[Section 2.6]{KacI}.

The two-way shift of geometric information $\lamod \longleftrightarrow \lahat$-mod occurs on the level of the ``small" Grassmannian parametrizing varieties $\grasstbd$ for the $\bd$-dimensional $\la$-modules with top $T$ and the analogous varieties $\grass^{\That}_{\bdhat}$, where $\That$ and $\bdhat$ are related to $T$ and $\bd$ as follows.  Given a semisimple $\la$-module $T$ with dimension vector $\bt = (t_1, \dots, t_n) \le \bd$, match it with the semisimple $\lahat$-module $\That$ that has dimension vector $(\bt, \mathbf 0) = (t_1, \dots, t_n, 0, \dots, 0)$. Moreover, pair the dimension vector $\bd$ of $Q$ with the dimension  vector $\bdhat = (\bt, \bd-\bt)$ of $\Qhat$; in other words, the last $n$ entries of $\bdhat$ amount to the dimension vector of $JM$ for any module $M$ in $\grasstbd$.
By $P$ and $\Phat$, we denote a $\la$-projective cover of $T$ and a $\lahat$-projective cover of $\That$, respectively.  According to Section \ref{sec7}, the automorphism groups of these projective modules act on the considered parametrizing varieties $\grass^{\bullet}_{\bullet}$ for $\la$-, resp., $\lahat$-modules, delineating isomorphism classes.  The $\autlap$-action on $\grasstbd$ boils down to an $\Aut_\la(T)$-action, due to the fact that the kernel of the natural map $\autlap \rightarrow \Aut_\la(T)$ acts trivially for $J^2 = 0$; analogously, the $\Aut_{\lahat}(\Phat)$-action on $\grass^{\That}_{\bdhat}$ boils down to an $\Aut_{\lahat}(\That)$-action.  It is readily checked that the obvious bijection $\Aut_\la(T) \rightarrow \Aut_{\lahat}(\That), \ g \mapsto \hat{g}$ is an isomorphism of algebraic groups.  Identifying these groups would allow us to view the following isomorphism of varieties as being equivariant under the relevant $\Aut$-action. 

\begin{proposition} \label{prop8.6} {\rm\cite[Proposition 5.3]{BCH}  $(J^2=0)$}  There is an isomorphism of varieties
$$\Phi^T_\bd:  \grasstbd \longrightarrow \grass^{\That}_{\bdhat}, \ \ C \longmapsto \Chat$$
such that $\Phi^T_\bd(g.C) = \hat{g}.\Chat$ for $g \in \Aut_\la(T)$.

This isomorphism yields a $1${\rm -}$1$ correspondence between the isomorphism classes of $\bd$-dimensional $\la$-modules with top $T$ on one hand and the isomorphism classes of $\bdhat$-dimen\-sio\-nal $\lahat$-modules with top $\That$ on the other, namely $M = P/C \mapsto \Mhat = \Phat/\Chat$.  The correspondence preserves and reflects direct sum decompositions in the following strong sense: 
 $M$ is a direct sum of submodules $M_r$ with dimension vectors $\bd^{(r)}$ for $1 \le r \le s$ precisely when $\Mhat$ is a direct sum of  submodules $\Mhat_r$ with dimension vectors $\bdhat^{(r)}$. $($Here $\bdhat^{(r)} = (\udim \Top M_r,\, \bd^{(r)}- \udim \Top M_r)$ and $\bd^{(r)} = (\bdhat_1^{(r)} + \bdhat_{n+1}^{(r)}, \dots, \bdhat_n^{(r)} + \bdhat_{2n}^{(r)})$.$)$ Moreover,  the direct summands $M_r$ of $M$ are indecomposable if and only if the same holds for the corresponding direct summands $\Mhat_r$ of $\Mhat$.
 
 Taken together, the correspondences $\Phi^T_\bd$ linking $\la$-modules to $\lahat$-modules for the various dimension vectors and tops induce bijections between the submodule lattices of the partners of any pair $(M, \Mhat)$.
 \end{proposition}
 
 As announced, this permits to transfer all of the generic information on the irreducible components of the varieties $\modlad$ to generic information on the irreducible components of $\Rep_{\bdhat}(\lahat)$, and vice versa.  We phrase the key points somewhat loosely.   

\begin{theorem} \label{thm8.7} {\rm\cite[Theorem 5.6]{BCH}}  $(J^2=0)$ Suppose $\C$ is an irreducible component of $\modlad$ with generic top $T$.  Moreover, let $\bdhat = (\udim T,\, \bd-\udim T)$ and $\That$ be as above.  Then:

{\bf (a)}  $\That$ is the generic top of the modules in $\Rep_{\bdhat}(\lahat)$.

{\bf (b)}  If $\bdhat = \bdhat^{(1)} + \cdots + \bdhat^{(s)}$ is the Kac  decomposition $($= canonical decomposition$)$ of $\, \bdhat$, then the Kac decomposition of $\bd$ relative to the component $\C$ is $\bd = \bd^{(1)} + \cdots + \bd^{(s)}$, where $\bd^{(r)} = (\bdhat_1^{(r)} + \bdhat_{n+1}^{(r)}, \dots, \bdhat_n^{(r)} + \bdhat_{2n}^{(r)})$.

{\bf (c)} The modules in $\C$ are generically indecomposable if and only if $\bdhat$ is a Schur root of the quiver $\Qhat$.   This, in turn, is equivalent to the condition that, generically, the modules $M$ in $\C$ satisfy $\End_\la(M) / \Hom_\la(M,JM) \cong K$.  

{\bf (d)} $\C$ contains a dense orbit if and only if $\Rep_{\bdhat}(\lahat)$ does. {\rm (Recall from Theorem \ref{thm5.3} that, in testing the condition for $\lahat$, one may use the Euler form.)}

{\bf (e)}  The vectors generically arising as dimension vectors of $\lahat$-submodules of objects in $\Rep_{\bdhat}(\lahat)$ are in one-to-one correspondence with the vectors generically arising as dimension vectors of $\la$-submodules of objects in $\C$.  More precisely: ${\mathbf \uhat} = (u_1,\dots,u_{2n})$ is attained on the submodule lattice of a generic module for $\Rep_{\bdhat}(\lahat)$ if and only if ${\mathbf u} = (u_1+u_{n+1}, \dots, u_n+ u_{2n})$ is attained on the submodule lattice of a generic module for $\C$.
\end{theorem}

\subsection*{\ref{sec8}.C.  Local algebras }

In this subsection we assume $\la$ to be a local truncated path algebra, meaning that the quiver $Q$ has the form 
$$\xymatrixrowsep{0.75pc}\xymatrixcolsep{0.3pc}
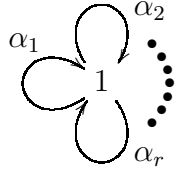
\xymatrix{
 & \ar@{{}{*}}@/^1pc/[dd] \\
1 \ar@'{@+{[0,0]+(-6,-6)} @+{[0,0]+(-15,0)} @+{[0,0]+(-6,6)}}^(0.7){\alpha_1}
\ar@'{@+{[0,0]+(-6,6)} @+{[0,0]+(0,15)} @+{[0,0]+(6,6)}}^(0.7){\alpha_2}
\ar@'{@+{[0,0]+(6,-6)} @+{[0,0]+(0,-15)} @+{[0,0]+(-6,-6)}}^(0.3){\alpha_r}  \\
 &
}$$
for some positive integer $r$.  Clearly, $\dim J/J^2 = r$, and dimension vectors $\bd$ amount to $K$-di\-men\-sions $d$ in this situation.  For $r = 1$,  $\la$ is a truncated polynomial ring in a single variable, and all of the varieties $\modlasd$ are trivially irreducible.  Otherwise, we find:

\begin{theorem} \label{thm8.8}  {\rm\cite[Main Theorem]{irredcompI} ($\la$ local truncated)}
Assume $\la$ has Loewy length $L+1$ and $Q$ has $r \ge 2$ loops.  Let $d$ be a positive integer.  

{\rm\bf(I)} If $d > L+1$ and $\SS$ is a $d$-dimensional semisimple sequence, the following  conditions are equivalent:
\begin{enumerate}
\item The closure $\overline{\laySS}$ is an irreducible component of $\modlasd$.
\item[(1$'$)] The closure $\, \overline{\biggrass\, \SS}$ is  an irreducible component of $\biggrass_d(\la)$.
\item[(2)]   $\dim \SS_l \le r \cdot \dim \SS_{l-1}$ and $\ \dim \SS_{l-1} \le r \cdot \dim \SS_l\,$ for $\, l \in \{1, \dots, L\}$.
\item[(3)] $\laySS \ne \varnothing$, and $(\SS_L, \SS_{L-1}, \dots, \SS_0)$ is the generic socle layering of the modules in $\laySS$.
\item[(4)]   $\SS = \SS(M)$ for some minimal pair $\bigl(\SS(M), \SS^*(M) \bigr)$  in $\Img(\Theta)$.
\end{enumerate}

\noindent Therefore, the irreducible components of $\modlasd$ are precisely the $\overline{\laySS}$ where $\SS$ traces the $\bd$-dimensional semisimple sequences satisfying $(1)-(4)$.

{\rm\bf (II)} If, on the other hand, $d \le L+1$, the variety $\modlasd$ is irreducible and, generically, its modules are uniserial. In this situation, conditions {\rm (1), (1$'$)} and {\rm (4)} are equivalent.
\end{theorem}

For $L = 1$ and $r = 2$, the irreducible components of the varieties $\modlasd$ had previously been determined by Donald and Flanigan \cite{DoFl}, as well as by Morrison \cite{Mor}.  For arbitrary choices of $r$, the case $L=1$ was covered in \cite[Theorem 3.12]{BCH}.

Condition (2) of the theorem permits us to list  --  without any computational effort  -- the irreducible components of $\modlasd$, tagged by their generic radical layerings.  Note moreover, that for local truncated path algebras, Theorem \ref{thm8.3} (which provides the generic socle layering for an irreducible component $\overline{\laySS}$) is superseded by the far simpler description given in condition (2) of Theorem \ref{thm8.8}.   Further generic properties of the modules in the components are derived in \cite[Section 4]{irredcompI}.

We conclude the discussion of the local case with an illustration.

\begin{example} \label{ex8.9} \cite[Example 4.5]{irredcompI}  Let $\la$ be the local truncated path algebra with $r = 3 = L+1$, and $d = 10$.  Then $\modlasd$ has precisely 17 irreducible components; one readily finds the eligible generic radical layerings via criterion (2) of Theorem \ref{thm8.8}. These layerings are displayed below, the numbers of bullets indicating the dimensions of the layers.
$$\xymatrixrowsep{0.1pc}\xymatrixcolsep{0.001pc}
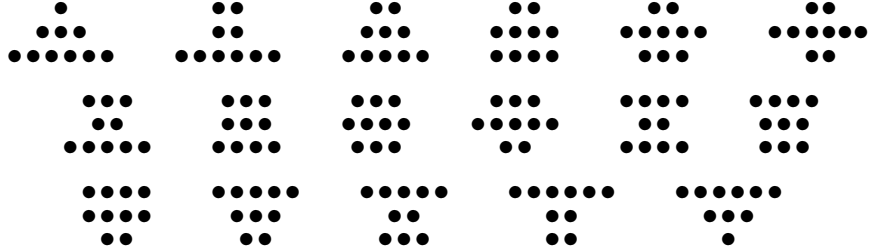
\xymatrix{
 &&\mbull &&& &&&& &&\dbull &\dbull && &&&& &\mbull &\mbull && &&&& &\dbull &\dbull & &&&& &\mbull &\mbull && &&&& &&\dbull &\dbull  \\
 &\mbull &\mbull &\mbull && &&&& &&\dbull &\dbull && &&&& &\dbull &\dbull &\dbull & &&&&\dbull &\dbull &\dbull &\dbull &&&&\dbull &\dbull &\dbull &\dbull &\dbull &&&&\dbull &\dbull &\dbull &\dbull &\dbull &\dbull  \\
\dbull &\dbull &\dbull &\dbull &\dbull &\dbull &&&&\dbull &\dbull &\dbull &\dbull &\dbull &\dbull &&&&\dbull &\dbull &\dbull &\dbull &\dbull &&&&\dbull &\dbull &\dbull &\dbull &&&& &\dbull &\dbull &\dbull & &&&& &&\dbull &\dbull  \\  \\
&&& &\dbull &\dbull &\dbull & &&&&\mbull &\mbull &\mbull & &&&&\mbull &\mbull &\mbull & &&&& &\dbull &\dbull &\dbull & &&&&\dbull &\dbull &\dbull &\dbull &&&&\dbull &\dbull &\dbull &\dbull  \\
&&& &\mbull &\mbull && &&&&\mbull &\mbull &\mbull & &&&&\dbull &\dbull &\dbull &\dbull &&&&\dbull &\dbull &\dbull &\dbull &\dbull &&&& &\dbull &\dbull & &&&&\mbull &\mbull &\mbull &  \\
&&& \dbull &\dbull &\dbull &\dbull &\dbull &&&&\dbull &\dbull &\dbull &\dbull &&&&\mbull &\mbull &\mbull & &&&& &\mbull &\mbull && &&&&\dbull &\dbull &\dbull &\dbull &&&&\mbull &\mbull &\mbull &  \\  \\
&&&& \dbull &\dbull &\dbull &\dbull &&&&\dbull &\dbull &\dbull &\dbull &\dbull &&&&\dbull &\dbull &\dbull &\dbull &\dbull &&&&\dbull &\dbull &\dbull &\dbull &\dbull &\dbull &&&&\dbull &\dbull &\dbull &\dbull &\dbull &\dbull  \\
&&&& \dbull &\dbull &\dbull &\dbull &&&& &\dbull &\dbull &\dbull & &&&& &\mbull &\mbull && &&&& &&\dbull &\dbull && &&&& &\mbull &\mbull &\mbull &  \\
&&&& &\dbull &\dbull & &&&& &\mbull &\mbull && &&&& &\dbull &\dbull &\dbull & &&&& &&\dbull &\dbull && &&&& &&\mbull &
}$$
 All of the components parametrize generically indecomposable modules and contain infinitely many $\GL_d$-orbits of maximal dimension in this example; see \cite[Corollary 4.3]{irredcompI}  
 for methods to check this.  
 \end{example}

 \subsection*{\ref{sec8}.D. Algebras based on acyclic quivers }

We continue to let $\la$ stand for a truncated path algebra with $J^{L+1} = 0$, but now we assume its quiver $Q$ to be acyclic.  Once again, the minimal values attained by the upper semicontinuous map $\Theta: \modlad \rightarrow \Seq(\bd) \times \Seq(\bd)$ are in bijective correspondence with the irreducible components of $\modlad$, via $(\SS, -) \mapsto \overline{\laySS}$.  But this time, we do not know of any shortcut to bypass comparisons of pairs $(\SS, \SS^*)$ to compile the list of sequences $\SS$ that give rise to the components of $\modlad$.  Our formulation of Theorem \ref{thm8.10} underlines the applicability of the satellite result, Theorem \ref{thm8.3}, to economize the sorting process.

\begin{theorem} \label{thm8.10} {\rm\cite[Main Theorem]{irredcompII} ($\la$ truncated, $Q$ acyclic)} Let  $\SS$ be a $\bd$-dimen\-sion\-al semisimple sequence and $\SS^*$ the generic socle layering of the modules in $\laySS$.  Then $\overline{\laySS}$ is an irreducible component of $\modlad$ if and only if $(\SS, \SS^*)$ is a minimal element of the set 
$$\{(\SShat, \SShat^*) \mid \SShat \in \Seq(\bd) \ \text{is realizable},\ \text{and}\ \ \SShat^*\  \text{is the generic socle layering of}\   \Rep\SShat\}.$$

The situation is symmetric in $\SS$ and $\SS^*$:  Whenever $(\SS, \SS^*)$ is a minimal element in the image of the detection map $\Theta$, then $\SS^*$ is the generic socle layering of ${\laySS}$, and $\SS$ is the generic radical layering of the modules with socle layering $\SS^*$.
\end{theorem}

\begin{kqmod*}
The information on truncated path algebras of acyclic quivers supplements the theory available for $KQ$, filling in generic data on the $\bd$-dimensional $\la$-modules of any fixed Loewy length as follows.  The generic radical layering $\SS$ of $\Rep_{\bd}(KQ)$ is directly available from $Q$ (see Proposition \ref{prop5.4}).  Suppose $L(\bd) := \max\{l\le L \mid \SS_l \ne 0\}$.  Since the set of those points in $\Rep_{\bd}(KQ)$ which encode modules of Loewy length $L(\bd) + 1$ is dense open in $\Rep_{\bd}(KQ)$, the reach of the Kac/Schofield results is limited to the $\bd$-dimensional modules of this Loewy length.  Excising this open subvariety leaves us with a copy of the variety $\Rep_\bd\bigl(KQ / \langle \text{the paths of length}
\ L(\bd)  \rangle \bigr)$; it has a dense open subset which parametrizes the $\bd$-dimensional $KQ$-modules with the reduced Loewy length $L(\bd)$.  Clearly, iterated excision of subvarieties of modules not annihilated by all paths of some fixed length thus supplies generic information about the $\bd$-dimensional $KQ$ modules of any Loewy length $m < L(\bd)$ by way of Theorems \ref{thm8.10}, \ref{thm4.6}, and the followup results proved in \cite[Section 3]{irredcompII}.
\end{kqmod*}   

\begin{example} \label{ex8.11} \cite[Example 5.1]{irredcompII} Let $\la_L = \CC Q/ \langle \text{the paths of length}\ L+1 \rangle$, where $Q$ is the quiver
$$\xymatrixrowsep{1pc}\xymatrixcolsep{3pc}
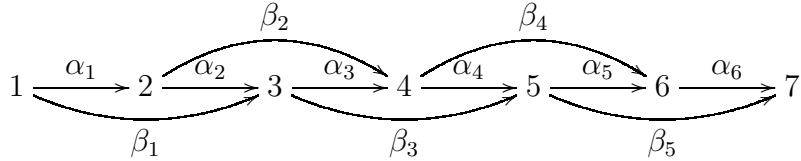
\xymatrix{
1 \ar[r]^{\alpha_1} \ar@/_1pc/[rr]_{\beta_1}  &2 \ar[r]^{\alpha_2} \ar@/^1.5pc/[rr]^{\beta_2}  &3 \ar[r]^{\alpha_3} \ar@/_1pc/[rr]_{\beta_3}   &4 \ar[r]^{\alpha_4} \ar@/^1.5pc/[rr]^{\beta_4}   &5 \ar[r]^{\alpha_5} \ar@/_1pc/[rr]_{\beta_5}   &6 \ar[r]^{\alpha_6}  &7
}$$
\noindent and $\bd = (1, 1, \dots, 1) \in \NN^7$.   Note that $K_0 = \QQ$ in this example.

$\bullet$ If $L= 6$, we have $\la_L = KQ$, whence $\Rep_{\bd}(\la_L)$ is irreducible.  In this case, the $\bd$-dimensional modules are generically uniserial with radical layering $(S_1, \dots, S_7)$. 

$\bullet$ For $L = 5$, Theorem \ref{thm8.10} shows the variety $\rep_{\bd}(\la_L)$ to have precisely $6$ irreducible components, all of them representing generically indecomposable modules.  They can be listed in terms of their generic modules which, by Theorem \ref{thm4.6}, are available from the generic radical layerings.  Graphs of these generic modules are displayed in Figure \ref{ex8.11}(a) below.
\begin{figure}[ht]
$$\xymatrixrowsep{0.7pc}\xymatrixcolsep{0.1pc}
\xymatrix{
1 \edge[dr] &&2 \edge[dl]\edge@/^1pc/[ddl]  &&&&& &1 \edge[d] \edge@/_1pc/[dd] &  &&&&& &1 \edge[d] \edge@/_1pc/[dd] &  &&&&& &1 \edge[d] \edge@/^1pc/[dd] &  &&&&& &1 \edge[d] \edge@/_1pc/[ddl] &  &&&&& &1 \edge[dl] \edge[dr]  \\
&3 \edge[d] \edge@/_1pc/[dd] &  &&&&& &2 \edge[d] \edge@/^1pc/[dd] &  &&&&& &2 \edge[d] \edge@/^1pc/[dd] &  &&&&& &2 \edge[d] \edge@/_1pc/[ddl] &  &&&&& &2 \edge[dl] \edge[dr] &  &&&&&2 \edge[dr] &&3 \edge[dl] \edge@/^1pc/[ddl]  \\
&4 \edge[d] \edge@/^1pc/[dd] &  &&&&& &3 \edge[d] \edge@/_1pc/[dd] &  &&&&& &3 \edge[d] \edge@/_1pc/[ddl] &  &&&&& &3 \edge[dl] \edge[dr] &  &&&&&3 \edge[dr] &&4 \edge[dl] \edge@/^1pc/[ddl]  &&&&& &4 \edge[d] \edge@/_1.25pc/[dd]  \\
&5 \edge[d] \edge@/_1pc/[dd] &  &&&&& &4 \edge[d] \edge@/_1pc/[ddl] &  &&&&& &4 \edge[dl] \edge[dr] &  &&&&&4 \edge[dr] &&5 \edge[dl] \edge@/^1pc/[ddl]  &&&&&  &5 \edge[d] \edge@/_1pc/[dd] &  &&&&& &5 \edge[d] \edge@/^1pc/[dd]  \\
&6 \edge[d] &  &&&&& &5 \edge[dl] \edge[dr] &  &&&&&5 \edge[dr] &&6 \edge[dl]  &&&&& &6 \edge[d] &  &&&&&  &6 \edge[d] &  &&&&& &6 \edge[d]  \\
&7 &  &&&&&6 &&7  &&&&& &7 &  && &&& &7 &  &&&&&  &7 &  &&&&& &7
}$$
\centerline{Figure \ref{ex8.11}(a)}
\end{figure}

Concerning the leftmost graph: The corresponding irreducible component of $\Mod_\bd(\la_5)$ equals the closure of $\laySS$, where $\SS = (S_1\oplus S_2, \, S_3, \, S_4, \, S_5, \, S_6, \, S_7)$. Generically, the modules in this component are of the form $G = (\la z_1 \oplus \la z_2) / C$, where $z_i = e_i z_i$ for $i=1,2$ and $C$ is the $\la$-submodule generated by $\alpha_1 z_1$, $\alpha_2 z_2 - x_1 \beta_1 z_1$, $\beta_2 z_2 - x_2 \alpha_3 \beta_1 z_1$, $\beta_3 \beta_1 z_1 - x_3 \alpha_4 \alpha_3 \beta_1 z_1$, $\beta_4 \alpha_3 \beta_1 z_1 - x_4 \alpha_5 \alpha_4 \alpha_3 \beta_1 z_1$, $\beta_5 \beta_3 \beta_1 z_1 - x_5 \alpha_6 \alpha_5 \alpha_4 \alpha_3 \beta_1 z_1$; here any choice of scalars $x_i \in \CC$ which are algebraically independent over $\QQ$ is permissible. 

$\bullet$ The case $L = 3$ is more interesting.  Using Theorem \ref{thm8.10}, one  finds that the variety $\Rep_\bd(\la_3)$ has precisely $28$ irreducible components. Of these, $12$ encode generically indecomposable modules; the modules in the remaining $16$ split into two indecomposable summands, generically.  The dimensions of the moduli spaces classifying the modules with the respective generic radical layerings (existent by \cite[Theorem 4.4]{classifying}) vary among $1$, $2$, $3$ for the different components.  In particular, none of the components contains a dense orbit.

We display $9$ of the components of $\Rep_\bd(\la_3)$ in Figure \ref{ex8.11}(b) below, again in terms of graphs of their generic modules. 
\begin{figure}[ht]
$$\xymatrixrowsep{0.5pc}\xymatrixcolsep{0.05pc}
\xymatrix{
\txt{(A)}  &&1 \edge[dl] \edge[dr]  & &&&&&\txt{(B)} &1 \edge[d] &&2 \edge[d] \edge[dll]  &&&&&\txt{(C)} &&1 \edge[dl] \edge[dr] &&&7 \drbl  \\
&2 \edge[d] \edge@/_1pc/[dd] &&3 \edge[d] \edge[dll]  &&&&& &3 \edge[dr] &&4 \edge[dl]  &&&&& &2 \edge[dr] &&3 \edge[dl] \edge@/^1pc/[ddll] &\bigoplus  \\
&4 \edge[d] &&5 \edge[d] \edge[dll]  &&&&& &&5 \edge[dl] \edge[dr]  & &&&&& &&4 \edge[dl] \edge[dr]  \\
&6 &&7  &&&&& &6 &&7  &&&&& &5 &&6  \\
1 \edge[d] &\txt{\hphantom{5}}&3 \edge[d] \edge[ddll]  &&&&&&1 \edge[d] &&&3 \edge[dl] \edge[dr]  & &&&&&&1 \edge[d] \edge@/_1pc/[dd] &&4 \edge[d] \edge@/^1pc/[dd]  \\
2 \edge[d] &&5 \edge[d] \edge@/^1pc/[dd]  &&&&&&2 &\bigoplus &4 \edge[dr] &&5 \edge[dl] \edge@/^/[ddl]  &&&&&&2 \edge[d] &\bigoplus &5 \edge[d] \edge@/_1pc/[dd]  \\
4 &&6 \edge[d] &&&&&& &&&6 \edge[d]  & &&&&&&3 &&6 \edge[d]  \\
&&7  &&&&&& &&&7  & &&&&&& &&7  \\
&1 \edge[dl] \edge[dr] &&4 \edge@/^/[dddl]  &&&&&&1 \edge[d] \edge@/_1pc/[dd] &&&5 \edge[d] \edge@/_0.4pc/[dddl]  &&&&&&1 \edge[dr] &&2 \edge[dl] \edge@/_1pc/[ddll] &6 \edge@/^/[dddl]  \\
2  &&3 \edge[dl] & &&&&&&2 \edge[d] \edge[dr] &&&7  &&&&&& &3 \edge[dl] \edge[dr]  \\
&5 \edge[dl] \edge[dr] && &&&&&&3 &4 \edge[dr] && &&&&&&4 &&5 \edge[d]  \\
7 &&6 & &&&&&& &&6 && &&&&& &&7
}$$
\centerline{Figure \ref{ex8.11}(b)}
\end{figure}
Let $\C_A$, $\C_B$, $\C_C$ denote the components whose generic modules have the graphs labeled (A), (B), (C), respectively.  Note that the generic radical layering of $\C_A$ is strictly smaller than that of $\C_B$, while the socle layerings are in reverse relation.  The generic socle layering of $\C_A$ is strictly smaller than that of $\C_C$, but the generic radical layerings of $\C_A$ and $\C_C$ are not comparable.  
\end{example}

While the process of comparing pairs $(\SS,\SS^*)$ may be streamlined in light of Theorem \ref{thm8.3}, for significantly larger examples, the need for bookkeeping will call for a computer program.

\subsection*{\ref{sec8}.E.  The general truncated case }

We now \emph{waive all conditions on $Q$}, but retain the hypothesis that $\la$ be truncated.  In general, the map $\Theta$ then fails to detect all irreducible components of the varieties $\modlad$.  The lowest Loewy length for this to occur is $4$; see \cite{BCHII}.

\begin{example} \label{ex8.12} \cite[Example 4.8]{irredcompI}  Let $\la = KQ / \langle \text{all paths of length\ } 4 \rangle$ and $\bd = (1,1,1,1)$, where $Q$ is the quiver 
$$\xymatrixrowsep{1pc}\xymatrixcolsep{2pc}
\xymatrix{
1 \ar[r]^{\alpha} &2 \ar@/^/[r]^{\beta} &3 \ar@/^/[l]^{\delta} \ar[r]^{\gamma} &4
}$$
Consider $\SS  =  (S_1, S_2, S_3, S_4)$ and $\SStilde  = (S_1 \oplus S_3, S_2 \oplus S_4, 0 , 0)$.  The varieties $\laySS$ and $\Mod \SStilde$ have generic modules $G$ and $\Gtilde$ determined by the the graphs displayed below. 
$$\xymatrixrowsep{1.5pc}\xymatrixcolsep{2pc}
\xymatrix{
1 \edge[d]_{\alpha}   &&1 \edge[d]_{\alpha} &3 \edge[d]^{\gamma}  \\
2 \edge[d]_{\beta}  &&2 \edge[ur]^(0.6){\delta} &4  \\
3 \edge[d]_{\gamma}  \\
4
}$$
In particular, we find that the generic socle layerings of $\laySS$ and $\Mod\SStilde$ are
$$\SS^* =  \SS^*(G) = (S_4, S_3, S_2, S_1) \ \ \ \text{and}\ \ \ \SStilde^* = \SS^*(\Gtilde) =   (S_2 \oplus S_4, S_1 \oplus S_3, 0, 0),$$
respectively, which shows $(\SS, \SS^*) < (\SStilde, \SStilde^*)$. It is also clear that $(\SS,\SS^*)$ is a minimal element of $\Img(\Theta)$, whence $\overline{\laySS}$ is an irreducible component of $\modlad$. On the other hand, $\delta \cdot G = 0$ while $\delta \cdot \Gtilde \ne 0$, whence the corresponding generic triples 
$$\bigl(\SS(G), \SS^*(G), \nullity_\delta(G) \bigr) \ \ \  \text{and} \ \ \ \bigl(\SS(\Gtilde), \SS^*(\Gtilde), \nullity_\delta(\Gtilde) \bigr)$$ 
are not comparable; here $\nullity_\delta X = \dim \ann_X (\delta)$.  In fact, the expanded map 
$$x \longmapsto (\SS(M_x), \SS^*(M_x), \nullity_\delta M_x)$$ 
achieves a minimal value on $\Mod \SStilde$, from which it follows that $\overline{\Mod\SStilde}$ is another irreducible component of $\modlad$.  As we saw, this component is not detected by the map $\Theta$ alone. For a more methodical treatment of Example \ref{ex8.12}, we will briefly revisit it at the end of the section.
\end{example}

The nullity argument we used in Example \ref{ex8.12} is very limited in scope; the same is true for tests combining $\Theta$ with general families of annihilator dimensions; see \cite[Example 6.1(b)]{GHZ}.  To amend the situation, we introduce a novel upper semicontinuous map, $\Gamma: \modlad \longrightarrow \NN$, which does not have any blind spots; that is, it always detects the generic radical layerings of the irreducible components when $\la$ is truncated.   The generic value of $\Gamma$ on any $\laySS$ is still algorithmically accessible from $Q$ and $L$ (see \ref{rem8.14}(1) below).  But the computations required are more labor-intensive than those called for by the $\Theta$-test.  Thus, compiling a list of the irreducible components of $\modlad$ is typically expedited by first locating the minimal values of $\Theta$; they always give rise to a subset of the set of irreducible components of $\modlad$, to be supplemented to the full collection by means of $\Gamma$ (see \ref{def8.16} for strategy).

Submodule filtrations of $\la$-modules continue to play the key role, but now we include filtrations beyond the radical and socle filtrations in order to probe $\modlad$ more thoroughly.  The next definition does not rely on the assumption that $\la$ is truncated. 

\begin{definition} \label{def8.13}  Let $\SS = (\SS_0, \dots, \SS_L)$ be any semisimple sequence with $\udim \SS = \bd$, and let $M \in \lamod$.  

$\bullet$ A filtration $M = M_0 \supseteq M_1 \supseteq \cdots \supseteq M_L \supseteq M_{L+1} = 0$ is said to be \emph{governed by} $\SS$ if $M_l / M_{l+1} = \SS_l$ for all $l$.  (Recall that we identify isomorphic semisimple modules.)

$\bullet$ $\Filt \SS  : = \{ x \in \modlad \mid \ \text{there exists a filtration of}\ M_x \ \text{governed by}\ \SS\}.$
\end{definition}

For our present purpose, we are only interested in filtrations governed by \emph{realizable} semisimple sequences.  (Criterion 8.1 may be used to list them.)  In general, there may be numerous non-realizable semisimple sequences governing filtrations of a module $M$.  

\begin{remarks} \label{rem8.14}  Recall that the present section is headed by the blanket hypothesis that $\la$ be truncated.  However, in the following comments, this assumption is only required where emphasized.  Let $\SS$ be a semisimple sequence with $\udim \SS = \bd$, and $M \in \lamod$.

{\bf (1)}  There is an alternative description of $\Filt \SS$, which permits to decide, for any point $x \in \modlad$, whether $x \in \Filt \SS$.  In case $\la$ is truncated, the decision relies exclusively on a similarity test with regard to the matrices in the entries of $x$ (see \cite[Lemma and Definition 3.6, and Section 5.B]{GHZ}). 

{\bf (2)}   $\Filt \SS$ is always nonempty.  Indeed, the module $\bigoplus_{0 \le l \le L}\SS_l$ clearly has a filtration governed by $\SS$.  

{\bf (3)}  The radical filtration $M = M_0 \supseteq JM \supseteq \cdots \supseteq J^{L+1}M = 0$ is the only filtration of $M$ governed by $\SS(M)$.  

{\bf (4)}  If $\SS$ governs a filtration of $M$, then $\SS \le \SS(M)$.

{\bf (5)} In Section \ref{sec9}, we will find that the sets $\Filt \SS$ are closed for any finite dimensional algebra $\la$.  In the present situation, closedness of these sets is part of the much stronger Theorem \ref{thmdef8.15}.  Moreover, we will see that part (2) of \ref{thmdef8.15} carries over to the general non-truncated case, but parts (1), (3) do not.  
\end{remarks}

\begin{theorem-definition} \label{thmdef8.15}  {\rm\cite[Theorems 4.3 and C]{GHZ} ($\la$ truncated)}

{\bf (1)}  If $\SS$ is a realizable semisimple sequence, then
$$\overline{\laySS} = \Filt \SS.$$

\noindent For $M \in \lamod$, 
let $\Gamma(M)$ be the number of those \emph{realizable} semisimple sequences which govern some filtration of $M$, and define 
$$\Gabull: \modlad \longrightarrow \NN, \ \ \ x \longmapsto \Gamma(M_x).$$

{\bf (2)}  The map $\Gabull$ is upper semicontinuous.  In particular, it is generically constant on the irreducible components of $\modlad$.

{\bf (3)} For a realizable semisimple sequence $\SS$ with $\udim \SS = \bd$, the following conditions are equivalent:
\begin{itemize}
\item $\overline{\laySS}$ is an irreducible component of $\modlad$.
\item $1 \in \Gabull(\laySS)$.
\item There exists a module $G$ in $\laySS$ with the property that the radical filtration $G \supseteq$ $JG \supseteq \cdots \supseteq J^{L+1} G$ is the only submodule filtration of $G$ which is governed by a realizable semisimple sequence.
\end{itemize}
\end{theorem-definition}

\begin{algorcomm} \label{def8.16} Clearly $1 \in \Gabull(\laySS)$ if and only if any generic module $G$ for $\laySS$ satisfies $\Gamma(G) = 1$.  One derives an effective algorithm for deciding whether the equivalent conditions of \ref{thmdef8.15}(3) are satisfied for $\SS$; see Remark \ref{rem8.14}(1) and \cite[Section 5.B]{GHZ}.  In the positive case, we call the sequence $\SS$ ``rigid".  However, in establishing the list of all rigid sequences with a given dimension vector $\bd$, exclusive reliance on the map $\Gabull$ is inefficient.  The following strategy takes advantage of the fact that $\overline{\Rep \SStilde} \subsetneqq \Filt \SS$ implies $\SS < \SStilde$, and also $\SS^* < \SStilde^*$ for the respective generic socle layerings. (The latter inequality is due to the fact that generic radical and socle layerings of the $\overline{\laySS}$ determine each other for truncated $\la$.)

Start by locating the set $\M_1$ of minimal pairs in $\Img(\Theta)$. The set $\A_1$ of first entries of the pairs in $\M_1$ then consists of rigid sequences; set $\B_1:= \varnothing$.  Next find the set $\M_2$ of minimal pairs in $\Img(\Theta) \setminus \M_1$. The above comment cuts down on the number of comparisons required to determine whether a given $\Rep \SStilde$ with $(\SStilde, \SStilde^*) \in \M_2$ is contained in $\Filt \SS$ for some $\SS \in \A_1$.  All those $\SStilde$ for which the answer is negative are assembled in a set $\A_2$, those for which the answer is positive are assembled in a set $\B_2$; clearly, the sequences in $\A_2$ are rigid, while those in $\B_2$ are not.  Now let $\M_3$ be the set of minimal pairs in $\Img(\Theta) \setminus \bigl(\M_1 \cup  \M_2 \bigr)$, and use the same remark to economize on the number of comparisons necessary to decide whether a given $\Rep \SStilde$ for a pair $(\SStilde, \SStilde^*) \in \M_3$ belongs to some $\Filt \SS$ with $\SS \in \A_1 \cup \A_2$.  A negative answer will lead to $\SStilde \in \A_3$, a positive answer to $\SStilde \in \B_3$.  Proceed inductively. 
\end{algorcomm}

\begin{example} \label{ex8.17}   Let $\la_0$ be the truncated path algebra associated with the algebra $\la$ of Example \ref{ex2.9}.  In particular, $\la_0$ has Loewy length $4$. For $\bd = (2,2)$, the variety $\Rep_\bd(\la_0)$ has precisely $5$ irreducible components.  Their generic modules are graphed in Figure \ref{ex8.17} below.  
\begin{figure}[ht]
$$\xymatrixrowsep{2.25pc} \xymatrixcolsep{0.9pc}
\xymatrix{
&1 \edge@/_/[dl]_{\alpha_2} \edge[d]^{\alpha_1} & &&&& &2 \edge@/_/[dl]_{\beta_2} \edge[d]^(0.4){\beta_1} \edge@/^1pc/[dr]^{\beta_3} & &&&& &&&1 \edge[dll]_{\alpha_2} \edge[drr]^{\alpha_1}   \\
2 &2 \save[0,0]+(0,3);[0,0]+(0,3) **\crv{~*=<2.5pt>{.} [0,0]+(3,3) &[0,0]+(3,0) &[0,0]+(3,-3) &[0,0]+(0,-3) &[0,-1]+(0,-3) &[0,-1]+(-3,-3) &[2,-1]+(-3,3) &[2,-1]+(0,3) &[2,0]+(0,3) &[2,0]+(3,3) &[2,0]+(3,0) &[2,0]+(3,-3) &[2,0]+(0,-3) &[2,-1]+(0,-3) &[2,-1]+(-5,-3) &[2,-1]+(-5,0) &[0,-1]+(-5,0) &[0,-1]+(-5,3) &[0,-1]+(0,3)} \restore
  \edge@/_1pc/[d]_(0.6){\beta_2} \edge[d]^(0.6){\beta_1} \edge@/^2pc/[d]^{\beta_3} & &&&&1  &1 \save[0,0]+(0,3);[0,0]+(0,3) **\crv{~*=<2.5pt>{.} [0,0]+(3,3) &[0,1]+(0,3) &[0,2]+(0,3) &[0,2]+(3,3) &[0,2]+(3,0) &[2,2]+(3,0) &[2,2]+(3,-3) &[2,2]+(0,-3) &[2,0]+(0,-3) &[2,0]+(-3,-3) &[2,0]+(-3,0) &[2,0]+(-3,3) &[2,0]+(0,3) &[2,1]+(0,3) &[2,1]+(7,3) &[0,1]+(7,-3) &[0,1]+(0,-3) &[0,0]+(0,-3) &[0,0]+(-3,-3) &[0,0]+(-3,0) &[0,0]+(-3,3)} \restore
  \save[0,0]+(0,5);[0,0]+(0,5) **\crv{~*=<2.5pt>{.} [0,0]+(4,5) &[0,0]+(4,0) &[0,0]+(4,-5) &[0,0]+(0,-5) &[0,-1]+(0,-5) &[0,-1]+(-3,-5) &[2,-1]+(-3,5) &[2,-1]+(0,5) &[2,0]+(0,5) &[2,0]+(4,5) &[2,0]+(4,0) &[2,0]+(4,-5) &[2,0]+(0,-5) &[2,-1]+(0,-5) &[2,-1]+(-5,-5) &[2,-1]+(-5,0) &[0,-1]+(-5,0) &[0,-1]+(-5,5) &[0,-1]+(0,5)} \restore
  \edge@/_1pc/[d]_(0.7){\alpha_2} \edge[d]^(0.6){\alpha_1} &1 &&&&&2 \edge@/^/[drr]^{\beta_1} \edge@/_/[drr]_{\beta_2} \edge@/_2.5pc/[drr]_{\beta_3} &&&&2 \edge@/_/[dll]_{\beta_1} \edge@/^/[dll]^{\beta_2} \edge@/^2.5pc/[dll]^{\beta_3}  \\
&1 \edge@/_1pc/[d]_(0.4){\alpha_2} \edge[d]^(0.4){\alpha_1} & &&&& &2 \edge@/_1pc/[d]_(0.4){\beta_2} \edge[d]^(0.4){\beta_1} \edge@/^2pc/[d]^(0.4){\beta_3}  & &&&& &&&1  \\
&2 & &&&& &1 &&  \\  
&&&&2 \edge@/_/[dl]_{\beta_2} \edge[d]^{\beta_1} \edge@/^1pc/[dr]^{\beta_3} & &&&&&&2 \edge@/_1pc/[d]_(0.4){\beta_2} \edge[d]^(0.4){\beta_1} \edge@/^1pc/[dr]^{\beta_3} &&&2 \edge@/_1pc/[d]_(0.4){\beta_2} \edge[d]^(0.4){\beta_1} \edge@/^1pc/[dr]^{\beta_3}  \\
&&&1 \horizpool{2} \edge@/_1pc/[dr]_{\alpha_1} \edge@/_3pc/[dr]_{\alpha_2} &1 \edge[d]_{\alpha_2} \edge@/^1pc/[d]^{\alpha_1} &1 &&&&&&1 \horizpool{3} &1 &&1 \save[0,0]+(0,5);[0,0]+(0,5) **\crv{~*=<2.5pt>{.} [0,1]+(0,5) &[0,1]+(4,5) &[0,1]+(4,0) &[1,1]+(4,4) &[1,1]+(0,4) &[1,1]+(-4,4) &[1,-3]+(4,4) &[1,-3]+(0,4) &[1,-3]+(-5,4) &[0,-3]+(-5,0) &[0,-3]+(-5,5) &[0,-3]+(0,5) &[0,-3]+(5,5) &[0,-3]+(5,0) &[0,-3]+(5,-8) &[0,0]+(-5,-8) &[0,0]+(-5,0) &[0,0]+(-5,5)} \restore
&1  \\
&&&&2 & &&&&&& &&& &
}$$
\centerline{Figure \ref{ex8.17}}
\end{figure}

We remark that the lower right graph, showing a generic module for $\Rep \SS^{(5)}$ where $\SS^{(5)} = (S_2^2, S_1^2, 0, 0)$, does not directly reflect the standard presentation provided by Theorem \ref{thm4.6}, but results from a slight simplification.  Moreover, on the model of the example concluding Section \ref{sec1}, we omitted some redundant edges in our graph of a generic module for the component $\overline{\Rep \SS^{(4)}}$, where $\SS^{(4)} = (S_2, S_1^2, S_2, 0)$.  Detail will follow.

\subsection*{Presentations of the displayed generic modules} In the presentations $P_0/C$ of the depicted generic modules (in each case $P_0$ is a $\la_0$-projective cover of the pertinent module), the expansions of the $\S$-critical paths, relative to a chosen skeleton $\S$ of $P_0/C$, involve coefficients in $\CC$ which are algebraically independent over $\QQ$.  We give detail regarding a generic module $G$ for $\Rep \SS^{(4)}$.  Consider the skeleton $\S = \{\bz , \beta_1 \bz, \beta_2 \bz, \alpha_1 \beta_1 \bz\}$ of $G$ in $P_0 = \la_0 \bz$.  On taking $\bz = e_2$, we obtain a projective presentation of $G$  of the form $\la_0 e_2/ C$, where $C$ is generated by $\beta_3 - x_1 \beta_1 - x_2 \beta_2$, $\alpha_2 \beta_2  - x_3 \alpha_1 \beta_1$, $\alpha_1 \beta_2 -  x_4 \alpha_1 \beta_1$, and $\alpha_2 \beta_1 = x_5 \alpha_1 \beta_1$ with scalars $x_1, \dots, x_5$ which are algebraically independent over $\QQ$; see Theorem \ref{thm4.6}.  The dependence relations tying the $\S$-critical paths $\alpha_2 \beta_3$ and $\alpha_1 \beta_3$ into the basis provided by $\S$ then arise as consequences and are not visually stressed in the graph of $G$.  To be specific: Setting $z = e_2+C$, we deduce that $\alpha_1 \beta_3 z = (x_1 + x_2 x_4) \alpha_1 \beta_1 z$ and $\alpha_2 \beta_3 z = (x_1 x_5 + x_2 x_3) \alpha_1 \beta_1 z$ in $G$.

\subsection*{Reasoning}  One first ascertains that the four components of Loewy length $> 2$ (two of which generically encode uniserial modules) are detected by the $\Theta$-test; indeed,  each of the eligible generic radical layerings arises as the first entry of a minimal element in $\Img(\Theta)$.  However, the sequence $\SS^{(5)} = (S_2^2, S_1^2, 0, 0)$ does not.  We use the $\Gamma$-test to establish the component status of $\overline{\Rep \SS^{(5)}}$:  From \ref{def8.16} we know that the only realizable semisimple sequences $\SS \neq \SS^{(5)}$ which potentially govern filtrations of a generic module $G$ for $\Rep \SS^{(5)}$ satisfy the inequality $\SS < \SS^{(5)}$, as well as $\SS^* < (\SS^{(5)})^*$ for the corresponding generic socle layerings.  The only sequence $\SS$ not ruled out by this constraint is $\SS = (S_2, S_1, S_2, S_1)$.  For a check that $G \notin \Filt \SS$, we point to a similar computation in \cite[Example 6.1(b)]{GHZ}. 

Next we use the $\Gamma$-test to ascertain that, for $\SStilde = (S_1^2, S_2^2, 0, 0)$, the generic value of $\Gabull$ on $\Rep \SStilde$ is at least $2$, whence $\SStilde$ fails to be generic for an irreducible component of $\modlad$.   Given that any generic module $\Gtilde$ for $\Rep \SStilde$ decomposes as shown in the related Example \ref{ex2.8}, it is clear that both $(S_1, S_1 \oplus S_2, S_2, 0)$ and $(S_1, S_2, S_1, S_2)$ govern filtrations of $\Gtilde$; the former sequence is not realizable, but the latter is. 

For the realizable semisimple sequences $\SS$ which were not addressed directly, the varieties $\Rep \SS$ are readily seen to be contained in $\Filt \SS^{(i)}$ for one or more of the displayed radical layerings $\SS^{(i)}$, $i \le 5$.  
\end{example}

\begin{return8.12*} We keep the previous notation.  As we already saw, the pair $(\SStilde, \SStilde^*)$ fails to be a minimal value of $\Theta$.  We will now use the $\Gamma$-test to show that the closure of $\Rep \SStilde$ is nonetheless an irreducible component of $\modlad$. Indeed, $\SS$ is the only realizable $\bd$-dimensional semisimple sequence strictly smaller than $\SStilde$, and the graph of the generic module $\Gtilde$ makes it evident that $\Gtilde$ does not have a filtration governed by $\SS$.  Thus $\Gamma(\Gtilde) = 1$.  

Inspection of the generic modules for the remaining varieties $\laySS'$ (there are  $6$ other realizable $\bd$-dimensional semisimple sequences) shows that $\modlad = \Filt(\SS) \cup \Filt(\SStilde) = \overline{\laySS} \cup \overline{\Rep\SStilde}$. Therefore $\overline{\laySS}$ and $\overline{\Rep\SStilde}$ are the only irreducible components of $\modlad$.
\end{return8.12*}

\section{Beyond truncated path algebras}
\label{sec9}

Now $\la = KQ/I$ denotes an arbitrary path algebra modulo an admissible ideal $I$, and $\la_0 = \la_{\text{trunc}}$ will be the associated truncated path algebra.  In this situation, the subvarieties $\laySS = \Rep_\la \SS$ of $\modlad$ may have arbitrarily many irreducible components.  Yet, some of the techniques developed for $\la_0$ in Section \ref{sec8} adapt to the general situation.  We spell out some detail and point to limitations of the approach to $\modlad$ by way of the closed immersion $\modlad \hookrightarrow \Rep_\bd(\la_0)$. 

First (in \ref{sec9}.A), we focus on the irreducible components of the varieties $\overline{\laySS}$.  Since radical layerings are generically constant, we already know that these components constitute a finite set of closed irreducible subvarieties of $\modlad$ which includes the irreducible components of $\modlad$, as $\SS$ traces the realizable semisimple sequences with dimension vector $\bd$.  We will see that this pivotal collection of subvarieties $\U$ is again accessible from $Q$ and $I$ (via the projective parametrizing varieties); each of the sets $\,\U$ arises in a representation-theoretic format, pinned down by a generic module.  In other words, each $\U$ is tagged by a module $G = G(\U)$ in $\U$ which combines all of the Morita-invariant generic properties of the modules in $\U$; cf.~Section \ref{sec4}.B for precision.  However, minimal projective presentations of the modules $G(\U)$ are not always as explicit as they are in the truncated case (cf.~Theorem \ref{thm4.6}).  Instead, they surface in the following format in general: $G(\U) = P / C$, where $P$ is a projective cover of $G(\U)$ and $C$ is given by way of generators involving a fixed ``path basis" of $P$, but now with coefficients subject to a system of polynomial equations; such a system (comparatively small) is concretely available from $Q$, generators for $I$, and a skeleton $\S$ with dimension vector $\bd$.  

In \ref{sec9}.B, we will single out results which carry over from the truncated to the general case, point blank.  In \ref{sec9}.C, we will follow with observations, provisional so far, on how to transfer algebra-specific information from $\Rep_\bd(\la_0)$ to $\modlad$, so as to expedite the process of selecting the irreducible components of $\modlad$ from the  set of irreducible subvarieties $\U$ which are now in the running for potential component status. In general, the set of eligible $\U$ is even more dramatically redundant than its incarnation in the truncated case, where it is $\{\overline{\laySS}\mid \SS \ \text{realizable}\}$.  In \ref{sec9}.D, finally, we will illustrate the strategy developed in the preceding subsections.

\subsection*{\ref{sec9}.A.  A finite set of irreducible subvarieties of $\modlad$ including all components:  the irreducible components of the varieties $\overline{\laySS}$}

To gain access to generic modules for the irreducible components of $\laySS$ (or, what amounts to the same, to the irreducible components of $\overline{\laySS}$), we further whittle down the varieties $\laySS$.  The patches of the open cover $(\layS)_\S$ of $\laySS$, where $\S$ runs through the skeleta with layering $\SS$, appear to offer themselves for the purpose (cf.~ \ref{sec4}.B).  However, the varieties $\layS$ are difficult to analyze.  On the other hand, they have almost-twin siblings in the projective scenario of Section \ref{sec7} which are far more amenable to analysis.  Since we are tackling the sequences $\SS \in \Seq(\bd)$ one at a time, it is moreover advantageous to work in the small projective setting, $\grassSS$, rather than the big, $\biggrassSS$.  We start by introducing the relevant subvarieties $\grassS$ of $\grassSS$, in turn open in $\grassSS$.  They do not coincide with the subvarieties of $\grassSS$ which correspond to the subvarieties $\layS$ of $\laySS$ under the bijection of Proposition \ref{prop7.1}, but are still smaller; indeed, in general, they are not stable under the $\autlap$-action on $\grassSS$, but only under the action of the unipotent radical of $\autlap$.  It is the $\autlap$-stable hull of any $\grassS$ in $\grassSS$ (evidently again open in $\grassSS$) that is the true twin of $\layS$ in the sense of \ref{prop7.1}.  

 Let  $\SS$ be a $\bd$-dimensional semisimple sequence and $\S$ a skeleton with layering $\SS$.  As in  \ref{sec4}.B, $P_0 = \bigoplus_{1 \le r \le t} \la \bz_r$ denotes a $\la_0$-projective cover of $T = \SS_0$, equipped with a full sequence $\bz_1, \dots, \bz_t$ of top elements, such that $\S$ consists of paths in $P_0$.  Then $P = \bigoplus_{1 \le r \le t} \la z_r$ is a $\la$-projective cover of $T$, provided $z_r$ is the image of $\bz_r$ under the canonical map $P_0 \rightarrow P_0 / IP_0 = P$.  Note that, as long as the chosen top elements $\bz_r$ remain fixed, $P_0$ contains only finitely many skeleta; a fortiori, there are only finitely many with layering $\SS$.

\begin{defthm} \label{defthm9.1} Given a skeleton $\S$ with layering $\SS$, define 
$$\grassS := \bigl\{C \in \grassSS \mid  \{p\, z_r + C \mid p\, \bz_r \in \S\} \ \text{is a basis for}\ P/C\bigr\}.$$

\noindent The subsets $\grassS$,  where $\S$ traces the skeleta with layering $\SS$, form an affine open cover of $\grassSS$.   An affine incarnation of $\grassS$ in the space $\AA^N$, where $N = \{(p\, \bz_r, q\, \bz_s) \mid p\, \bz_r \in \S,\ q\, \bz_s\ \S\text{-critical},\ \term(p)=\term(q),\ \len(p) \ge \len(q)\/ \}$ may be obtained from $Q$, generators for $I$, and $\S$ by way of an implemented algorithm.
\end{defthm}

See \cite[Lemma 3.8]{classifying} or \cite[Corollary 3.8]{hier} for openness of the $\grassS$; the best reference for the fact that the $\grassS$ are affine varieties is \cite[Theorem 3.12]{hier}.  The proof of  \cite[Theorem 3.12]{hier} also provides the theoretical underpinnings for the (straightforward) algorithm to compute the $\grassS$ in their affine coordinates.  In tandem, this algorithm actually yields minimal projective presentations of generic modules for the various components of any $\grassS$. 
A computer-implementation (without proof) can be found in \cite{codes}.  Yet, for examples of moderate size, the lightweight manual computation is less laborious than feeding the pertinent data into the program.  

For emphasis, we restate, in more detail, a fact already encountered in  \ref{defobs4.4}.  A skeleton $\S$ is called \emph{realizable} in case $\grassS \ne \varnothing$.  By the preceding remarks this amounts to the same as nonemptiness of $\layS$.  (The decision whether $\S$ is realizable comes as a byproduct of the mentioned algorithm.)  

\begin{corterm} \label{corterm9.2} {\bf (1)}  Let $\SS \in \Seq(\bd)$.  Every irreducible component of $\laySS$ intersects some $\layS$ in a dense open set; here $\S$ traces the ``realizable" skeleta with layering $\SS$, i.e., those $\S$ for which $\layS \ne \varnothing$ $($equivalently $\grassS \ne \varnothing)$. Hence the set of irreducible components of $\overline{\laySS}$ equals
$$\text{ {\rm Comp}}_\la\, \SS :=  \{\overline{\D} \mid \D \ \text{is an irreducible component of}\ \layS \ \text{for some} \ \S \ \text{with layering}\ \SS\}.$$

{\bf (2)}  All irreducible components of $\modlad$ are contained in 
$\ \bigcup_{\, \SS \in \Seq(\bd)} \text{{\rm Comp}}_\la \, \SS$.
\end{corterm} 

\subsection*{Conclusion} Once again, we have converted the task of representation-theoretically characterizing the irreducible components of $\modlad$ from the data $Q$, $I$, $\bd$ into a sorting problem.  However, in general, the problem of separating the ``grain from the chaff" by means of generic modules for the varieties collected in $\bigcup \text{{\rm Comp}}_\la\, \SS$ is much more complex than in the truncated case.  We do not expect a recipe leading to a meaningful overarching solution.  Rather, it appears promising to deal with the combinatorial difficulties by specializing to algebras of particular interest, such as group algebras of elementary abelian $p$-groups; for these group algebras, quiver presentations are immediate (see \cite{FrPe, FPS}, for instance, to appreciate the role they play in the theory of group representations). To date, the component problem is not even fully resolved for monomial algebras. 

\subsection*{\ref{sec9}.B. Facts which carry over from truncated to general algebras }

Let $\SS$ be a $\bd$-dimensional semisimple sequence.  We only deal with $\la$-modules in this section, whence the notation $\laySS$ is unambiguous. As in \ref{ex8.12}, one defines what it means for a submodule filtration of a $\la$-module $M$ to be \emph{governed by $\,\SS$}, and as before one denotes by $\Filt \SS$ the set of all those $x \in \modlad$ for which $M_x$ has a filtration governed by $\SS$.  As we already emphasized, most of the remarks in \ref{rem8.14} carry over to the general case, as do the definitions of $\Gamma(M)$ and $\Gabull: \modlad \rightarrow \NN$ in \ref{thmdef8.15}.  On the other hand, the equality $\overline{\laySS} = \Filt \SS$, which holds for all realizable semisimple sequences over truncated path algebras, needs to be replaced by an inclusion as follows.  

\begin{theorem} \label{thm9.3} {\rm\cite[Theorem 3.8, Corollary 3.11]{GHZ}} $\Filt \SS$ is always closed in $\modlad$ $($irrespective of whether or not $\SS$ is realizable$)$.
In particular, $\overline{\laySS} \subseteq \Filt \SS$.

Consequently, the map $\Gabull:  \modlad \rightarrow \NN$, $x \mapsto \Gamma(M_x)$ is upper semicontinuous.
\end{theorem}

Caveat:  Not only may $\overline{\laySS}$ contain arbitrarily many irreducible components of the ambient variety $\modlad$, the set difference $\Filt \SS \setminus \overline{\laySS}$ may contribute to the components of $\modlad$ as well.  Indeed, the closure of this difference may in turn include an arbitrarily high number of components of $\modlad$.  For a small instance, see Example \ref{ex9.8}.

The final statement of Theorem \ref{thm9.3} has an immediate offshoot regarding the component problem.

\begin{corollary} \label{cor9.4} {\rm\cite[Corollary 3.11]{GHZ}}  Whenever $\D$ is an irreducible component of some $\laySS$ with $1 \in \Gabull(\D)$, the closure $\overline{\D}$ is an irreducible component of $\modlad$.

In particular, $\overline{\D}$ is an irreducible component of $\modlad$ in case $\D$  contains a uniserial module.
\end{corollary}

Caveat: In general this sufficient condition fails to be necessary; see Example \ref{ex9.8}. 

In light of semicontinuity of $\Theta: \modlad \rightarrow \Seq(\bd) \times \Seq(\bd)$, we moreover obtain: Whenever $\SS \in \Seq(\bd)$ is minimal realizable, \emph{all} irreducible components of $\laySS$ close off to irreducible components of $\modlad$.   However, beyond the minimal case, separate maximality tests are required for the individual irreducible components of $\overline{\laySS}$.   Indeed, some of the irreducible components of $\overline{\laySS}$ may remain maximal irreducible in $\modlad$, while others are embedded in strictly larger irreducible subsets; see Example \ref{ex9.9}.

\subsection*{\ref{sec9}.C. Interplay between $\modlad$ and $\Rep_\bd(\la_{\text{trunc}})$ }

We continue to abbreviate $\la_{\text{trunc}}$ to $\la_0$.  Since we are now moving back and forth between $\Rep_\bd(\la_0)$ and $\modlad$, we will use subscripts to distinguish subvarieties of $\modlad$ from varieties in $\Rep_\bd(\la_0)$ in order to avoid ambiguities.  Clearly, the semisimple $\la$-modules coincide with the semisimple $\la_0$-modules; hence we need not make a distinction between semisimple sequences over $\la$ and $\la_0$.   

As we already pointed out in \ref{sec9}.B: To pin down the components of $\modlad$, it does not suffice to locate the sequences $\SS \in \Seq(\bd)$ with the property that $\SS$ is the generic radical layering of \emph{some} irreducible component of $\modlad$.  However, the following semisimple sequences may be dealt with in one fell swoop:

\begin{observation} \label{obs9.5}  Suppose a generic $\la_0$-module $G$ for some irreducible component of $\Rep_{\bd}(\la_0)$ is defined over $\la$, i.e., $IG = 0$.  Then $\overline{\Rep_\la \SS(G)}$ is an irreducible component of $\modlad$, and the $\la$-module $G$ is generic for this component.  
\end{observation}

Indeed, under the hypothesis of  \ref{obs9.5}, the closure of $\Rep_\la \SS(G)$ in $\modlad$ coincides with the closure of $\Rep_{\la_0} \SS$ in $\Rep_\bd(\la_0)$.

We follow with a more systematic approach to pulling information on the components of $\modlad$ from the full collection of components of $\Rep_\bd(\la_0)$.
Since $\modlad$ is a subvariety of $\Rep_\bd(\la_0)$,  each irreducible component of $\modlad$ is contained in an irreducible component of $\Rep_\bd(\la_0)$.  Suppose
$$\overline{\Rep_{\la_0} \SS^{(1)}} = \Filt_{\la_0} \SS^{(1)}, \dots,  \overline{\Rep_{\la_0} \SS^{(m)}}  = \Filt_{\la_0} \SS^{(m)}$$
are the distinct irreducible components of $\Rep_\bd(\la_0)$. 
Moreover, let $\C$ be an irreducible component of some $\Mod_\la \SS$ with generic module $G$.  First, one determines which among the $\SS^{(j)}$ govern a filtration of $G$;  these are the ones for which $\C \subseteq \Filt_\la \SS^{(j)}$.  Suppose the pertinent sequences are $\SS^{(1)}, \dots, \SS^{(r)}$.

\begin{observation} \label{obs9.6} {\rm\cite[Observation 6.5]{GHZ}}  $\overline{\C}$ is an irreducible component of $\modlad$ if and only if $\,\overline{\C}$ is maximal irreducible in $\Filt_\la \SS^{(j)}$ for $1\le j \le r$.
\end{observation}

This observation leads to a lower bound for the number of irreducible components of $\modlad$.  Evidently, sharpness of this bound is witnessed by any truncated path algebra $\la$, as well as by the dimension vectors $\bd$ whose modules are annihilated by $J^2$; indeed, in that case, the $\bd$-dimensional $\la_0$-modules are all defined over $\la$.    

\begin{corollary} \label{cor9.7} {\rm\cite[Corollary 6.6]{GHZ}} Again, let $\bd$ be a dimension vector, and adopt the above notation for the irreducible components of $\Mod_\bd(\la_0)$. 
Moreover, set
$$\A_j : = \Mod_\bd(\la_0)\  \setminus \bigcup_{1 \le i \le m,\ i \ne j} \Filt_{\la_0} \SS^{(i)} \ \ \ \ \ \text{for} \ 1 \le j \le m.$$
Then the number of irreducible components of $\modlad$ is bounded from below by the number $b$ of those $j$ with the property that $\A_j \cap \modlad \ne \varnothing$. 
\end{corollary}

For an example illustrating  \ref{cor9.7}, where $\modlad \ne \Rep_\bd(\la_0)$, we refer to Example \ref{ex9.8}(2).  In that instance, each of  the components of the $\bd$-dimensional $\la$-modules of Loewy length $> 2$ is properly contained in precisely one of the closures $\overline{\A_j}$.

 \subsection*{\ref{sec9}.D. Illustration}
 
 We first return to the algebra $\la$ of Example \ref{ex2.9}, to determine the irreducible components of $\Rep_{(2,2)} \la$.
 Then we will consider variants of $\la$ (each obtained from $\la$ by modding out one or two additional monomial relations) and track the changes in the number and generic behavior of the components entailed by the modifications.  In particular, the outcome will serve to back the caveats of \ref{sec9}.B,C.
 
 \begin{example} \label{ex9.8} Let $\la = \CC Q/I$ be as in Example \ref{ex2.9}, and $\bd = (2,2)$.  Recall that $Q$ consists of two vertices, $e_1$ and $e_2$, next to five arrows, two  from $e_1$ to $e_2$ labeled $\alpha_j$, and three in the opposite direction labeled $\beta_i$; the ideal $I$ is generated by $\beta_i \alpha_j$  for $i \ne j$, $\alpha_1 \beta_2$, and all paths of length $4$.
In Example \ref{ex8.17}, we discussed the irreducible components of $\Rep_\bd (\la_0)$, where $\la_0$ is the associated truncated path algebra of $\la$; we found exactly $5$ components in that case.   

\subsection*{The components of $\modlad$} This variety has precisely $8$ irreducible components.  Generic modules for $7$ of them, $\C_1, \dots, \C_7$, are graphed in Figure \ref{ex9.8}(a) below.    
\begin{figure}[ht]
$$\xymatrixrowsep{1.9pc} \xymatrixcolsep{0.8pc}
\xymatrix{
1 \dropup{G_1} \edge@/_2pc/[ddd]_{\alpha_2} \edge[d]^{\alpha_1} &&&&1 \dropup{G_2} \edge@/_2pc/[ddd]_{\alpha_1} \edge[d]^{\alpha_2} &&&&1 \dropup{G_3} \edge[dl]_{\alpha_2} \edge[dr]^{\alpha_1} & &&&1 \edge@/_/[d]_{\alpha_1} \edge@/^/[d]^{\alpha_2} &\save+<2.5ex,4ex> \drop{G_4} \restore \save+<2.5ex,-4ex> \drop{\bigoplus} \restore&&1 \edge@/_/[d]_{\alpha_1} \edge@/^/[d]^{\alpha_2}   \\
2 \edge[d]^{\beta_1} &&&&2 \edge[d]^{\beta_2} &&&2 \edge[dr]_{\beta_2} &&2 \edge[dl]^{\beta_1} & & &2 &&&2  \\
1 \edge[d]^{\alpha_1} \edge@/^1.5pc/[d]^{\alpha_2} &&&&1 \edge[d]^{\alpha_2} &&& &1  \\
2 &&&&2  \\
2 \dropup{G_5} \edge@/_2pc/[ddd]_{\beta_2} \edge[d]^(0.45){\beta_1} \edge@/^1pc/[dr]^{\beta_3} & &&&&&&2 \dropup{G_6} \edge@/_/[dl]_{\beta_1} \edge[d]^(0.35){\beta_2} \edge@/^1pc/[dr]^{\beta_3} & &&& &&&2 \dropup{G_7} \edge@/_/[dl]_{\beta_2} \edge[d]^(0.4){\beta_1} \edge@/^1pc/[dr]^{\beta_3}    \\
1 \save[0,0]+(0,3);[0,0]+(0,3) **\crv{~*=<2.5pt>{.} [0,0]+(3,3) &[0,1]+(0,3) &[0,1]+(3,3) &[0,1]+(3,0) &[2,1]+(3,0) &[2,1]+(3,-3) &[2,1]+(0,-3) &[2,0]+(0,-3) &[2,0]+(-3,-3) &[2,0]+(-3,0) &[2,0]+(-3,3) &[2,0]+(0,3) &[2,1]+(0,3) &[0,1]+(0,-3) &[0,0]+(0,-3) &[0,0]+(-3,-3) &[0,0]+(-3,0) &[0,0]+(-3,3)} \restore
 \edge[d]^(0.6){\alpha_1} &1& &&&&1  &1 \save[0,0]+(0,3);[0,0]+(0,3) **\crv{~*=<2.5pt>{.} [0,0]+(3,3) &[0,1]+(0,3) &[0,1]+(3,3) &[0,1]+(3,0) &[2,1]+(3,0) &[2,1]+(3,-3) &[2,1]+(0,-3) &[2,0]+(0,-3) &[2,0]+(-3,-3) &[2,0]+(-3,0) &[2,0]+(-3,3) &[2,0]+(0,3) &[2,1]+(0,3) &[0,1]+(0,-3) &[0,0]+(0,-3) &[0,0]+(-3,-3) &[0,0]+(-3,0) &[0,0]+(-3,3)} \restore
  \save[0,0]+(0,5);[0,0]+(0,5) **\crv{~*=<2.5pt>{.} [0,0]+(4,5) &[0,0]+(4,0) &[0,0]+(4,-5) &[0,0]+(0,-5) &[0,-1]+(0,-5) &[0,-1]+(-3,-5) &[2,-1]+(-3,5) &[2,-1]+(0,5) &[2,0]+(0,5) &[2,0]+(4,5) &[2,0]+(4,0) &[2,0]+(4,-5) &[2,0]+(0,-5) &[2,-1]+(0,-5) &[2,-1]+(-5,-5) &[2,-1]+(-5,0) &[0,-1]+(-5,0) &[0,-1]+(-5,5) &[0,-1]+(0,5)} \restore
  \edge[d]_(0.7){\alpha_2} &1 &&&&&1 \horizpool{2} &1 \edge@/_1.5pc/[d]_(0.55){\alpha_1} \edge[d]_(0.55){\alpha_2} &1 \edge[dl]^{\alpha_1} \edge@/^1.5pc/[dl]^(0.35){\alpha_2}   \\
2 \edge[d]^(0.4){\beta_1} & &&&&& &2 \edge[d]_(0.3){\beta_2} & &&&& &&2   \\
1 & &&&&& &1 &&  \\
& &&&&& &\dropvert{0}{\txt{Figure \ref{ex9.8}(a)}}
}$$
\end{figure}

\noindent The additional component $\C_8 = \overline{ \Rep_\la (S_2^2, S_1^2,0,0)}$ has generic module $G_8$ as graphed in Example \ref{ex8.17}.  Indeed, since the modules in $\C_8$ are annihilated by $J^2$, the variety $\C_8$ coincides with the closure $\overline{\Rep_{\la_0}\SS(G_8)}$ in $\Rep_\bd(\la_0)$ by Observation \ref{obs9.5}.  The same is true for the variants $\la_i$ of $\la$ presented below, and consequently we will exclude the semisimple sequence $\SS(G_8)$ from further discussion. 

We emphasize that the generic module for $\overline{\Rep_{\la_0} (S_2,S_1^2,S_2,0)}$ presented in Example \ref{ex8.17}  --  call it $G_7(\la_0)$  --  has a graph coinciding with that of $G_7$ above, even though $G_7(\la_0)$ is not defined over $\la$. Indeed, the graphs are only optimally informative in the presence of quiver and relations for the underlying algebra.  A comparison of projective presentations of $G_7 = G_7(\la)$ and $G_7(\la_0))$ will follow.

 Justification of the diagram: It is easy to check that, for all but one of the $\bd$-dimensional semisimple sequences $\SS'$ which are not among the $\SS(G_i)$, we have $\laySS' \subseteq \overline{\Rep \SS(G_j)}$ for some $j$.  As for the outsider sequence $\SS = (S_1\oplus S_2, S_1\oplus S_2, 0,0)$, we will see that $\laySS \subseteq \overline{\laySS(G_7)}$. This will permit us to conclude that the components of the $\laySS'$ with $\SS' \ne \SS(G_i)$ for $i=1,\dots,8$ do not contribute to the irreducible components of $\modlad$. 
 
 To prove that $\Rep_\la \SS \subseteq \overline{\Rep_\la \SS(G_7)}$, we observe that $\Rep_\la \SS = \Rep_{\la_0} \SS$ is irreducible with generic module $G(\SS) = (\la z_1/C_1) \oplus (\la z_2/C_2)$, where $C_1 = \la (\alpha_2 z_1 - y_1\alpha_1 z_1)$ and $C_2 = \la (\beta_2 z_2 - y_2 \beta_1 z_2) + \la (\beta_3 z_2 - y_3 \beta_1 z_2)$ with scalars $y_i \in \CC$ which are algebraically independent over $\QQ$. Both $G_7= G_7(\la)$ and $G_7(\la_0)$ have the following skeleton $\S$:
 $$\xymatrixrowsep{1.5pc} \xymatrixcolsep{0.8pc}
\xymatrix{
2 \edge[d]_{\beta_1} \edge[dr]^{\beta_3}  \\
1 \edge[d]_{\alpha_1} &1  \\
2
}$$
But the algorithm referenced in \ref{sec9}.A yields $\grass_\la (\S) \cong \AA^4$. More precisely, it shows that $\Rep_\la (\S)$ consists of the modules having a presentation as follows:
$\la e_2/C$, where $C$ is generated by 
$$\beta_2 - x_1 \beta_3 + x_1 x_3 \beta_1 \,, \ \alpha_2 \beta_1 - x_2 \alpha_1 \beta_1 \,, \ \alpha_1 \beta_3 - x_3 \alpha_1 \beta_1 \,, \ \alpha_2 \beta_3 - x_4 \alpha_1 \beta_1 \,;$$
here $(x_1,x_2,x_3,x_4)$ traces $\AA^4$. An elementary computation yields that there exists a module $M$ in $\Rep_\la(\S)$, together with a submodule $U \subseteq M$, such that $M/U \cong \la z_2/C_2$ and $U \cong \la z_1/C_1$. In particular, $M$ degenerates to $G(\SS)$, which places $G(\SS)$ into the closure of $\Rep_\la \SS(G_7)$. By contrast, the variety $\grass_{\la_0}(\S) \cong \AA^5$; in fact, the generic number of parameters of $\Rep_{\la_0} \SS(G_7)$ is $5$.
 
 The uniserial modules $G_1$, $G_2$ (resp., $G_5$, $G_6$) are clearly generic for irreducible components of $\Rep \SS(G_1)$ (resp., $\Rep \SS(G_5)$) that close off to irreducible components of $\modlad$ by Corollary \ref{cor9.4}.  To see that there are precisely two components with generic radical layering $\SS(G_1) = \SS(G_2)$, observe that there are only three realizable skeleta $\S$ with this layering: two of them, $\S_1$ and $\S_2$, are the skeleta of $G_1$, the remaining one, $\S_3$, is a skeleton of $G_2$.  The algorithm  mentioned in \ref{sec9}.A yields $\grassS \cong \AA^1$ for each of them, where $\grass(\S_1) \cap \grass(\S_2) \ne \varnothing$, while $\grass(\S_3)$ does not intersect the $\grass(\S_i)$ for $i=1,2$. This justifies the claim that $\Rep \SS(G_1)$ has precisely $2$ irreducible components, with generic modules as displayed. Similar considerations apply to $\Rep \SS(G_5) = \Rep \SS(G_6)$. This shows that our list includes all components of $\modlad$ containing modules of Loewy length $4$.

Since the semisimple sequences $(S_1,S_2,S_2,S_1)$ and $(S_2,S_1,S_1,S_2)$ are not realizable, we find that $\Gamma(G_3) = \Gamma(G_7) = 1$. Therefore, the irreducible components of $\Rep \SS(G_3)$ and $\Rep \SS(G_7)$ containing $G_3$ and $G_7$, respectively, close off to irreducible components of $\modlad$ by Corollary \ref{cor9.4}. The two realizable skeleta with layering $\SS(G_3)$ lead to the same variety $\grassS \cong \AA^1$, which shows $\Rep \SS(G_3)$ to be irreducible. Hence $\overline{\Rep \SS(G_3)}$ is an irreducible component of $\modlad$.  Analogously, so is $\overline{\Rep \SS(G_7)}$. 

A comparison of $\Theta$-values confirms that $G_4$ does not belong to $\overline{\Rep \SS(G_i)}$ for $i=3,5,6,7,8$, and since $\nullity_{\alpha_i}(G_4) = 2$ we see that $G_4$ is not in $\overline{\Rep \SS(G_i)}$ for $i=1,2$.  Moreover, irreducibility of the subvariety $\Rep \SS(G_4)$ of $\modlad$ is guaranteed by the fact that it equals $\Rep_{\la_0} \SS(G_4)$.  Therefore $\overline{\Rep \SS(G_4)}$ is an irreducible component of $\modlad$ as well.

Thus, the exhibited irreducible components of $\modlad$ are the only ones. Clearly, they are distinct.

We conclude by deriving  information regarding some of the issues addressed in Section \ref{sec9}.B.
In particular, $G_4$ belongs to $\bigl(\Filt \SS(G_1)\bigr) \setminus \overline{\Rep \SS(G_1)}$. In light of \cite[Corollary 4.5]{BHT}, this attests to the fact that the closure of $\Filt \SS(G_1) \setminus \overline{\Rep \SS(G_1)}$ may contribute to the collection of irreducible components of $\modlad$, as was pointed out after Theorem \ref{thm9.3}.  Analogously, $\Rep \SS(G_4)$ is contained in the closure of $\bigl(\Filt \SS(G_2)\bigr) \setminus \overline{\Rep \SS(G_2)}$.  Finally, in light of $\Gamma(G_4) \ge 2$, the component $\overline{\Rep \SS(G_4)}$ of $\modlad$ provides a counterexample to the converse of Corollary \ref{cor9.4}. 

\subsection*{Two variants of the algebra $\la$}  In the first of the upcoming variants, we add one more monomial relation to the presentation of the algebra $\la$, to arrive at an algebra $\la_1$ with the property that the component $\overline{\Rep_\la \SS(G_7)}$ disappears over $\la_1$, while $\Rep_{\la_1} \SS(G_7)$ is still irreducible (in particular nonempty).  Variant 2 is the example announced after Corollary \ref{cor9.7}.     

{\bf (1)} The changed picture for the factor algebra $\la_1 =KQ/I_1$, where $I_1 = I + \langle \alpha_1 \beta_3 \rangle$:  The variety $\Rep_\bd(\la_1)$ has only $7$ irreducible components.  The generic modules $G_1$, $G_2$, $G_3$, $G_4$, $G_6$, $G_8$ are defined over  the factor algebra $\la_1$ of $\la$ and hence the corresponding components of $\modlad$ remain intact as components of $\Rep_\bd(\la_1)$.  Generic modules $G'_5$ and $G'_7$ for $\Rep_{\la_1} \SS(G_5)$ and $\Rep_{\la_1} \SS(G_7)$, respectively, are shown in Figure \ref{ex9.8}(b) below. 
\begin{figure}[ht]
$$\xymatrixrowsep{1.75pc} \xymatrixcolsep{1pc}
\xymatrix{
2 \dropup{G'_5} \edge@/_2pc/[ddd]_{\beta_2} \edge[d]^(0.6){\beta_1} \edge@/^2.5pc/[ddd]^{\beta_3} &&&&&& &2 \dropup{G'_7}  \edge@/_/[dl]_{\beta_2} \edge[d]^(0.4){\beta_1} \edge@/^1pc/[dr]^{\beta_3}   \\
1 \edge[d]^{\alpha_1} &&&&&&1 \horizpool{2} &1 \edge[d]_(0.55){\alpha_2} &1 \edge[dl]^{\alpha_2}  \\
2 \edge[d]^(0.4){\beta_1} &&&&&& &2  \\
1
}$$
\centerline{Figure \ref{ex9.8}(b)}
\end{figure}
Observe that $G_5'$ is uniserial, and is therefore generic for an irreducible component of $\Rep_\bd(\la_1)$ by Corollary \ref{cor9.4}.   By contrast, one readily checks that $G'_7$ belongs to the irreducible component of $\Rep_\bd(\la_1)$ represented by $G_6$; to verify this, note that the socle layering of $G'_7$ is $(S_1 \oplus S_2, S_1, S_2, 0)$.

{\bf (2)} For the factor algebra $\la_2 = KQ/I_2$ of $\la$, where $I_2 = I + \langle \alpha_2 \beta_2 , \beta_2 \alpha_2 \rangle$, the variety $\Rep_\bd(\la_2)$ has precisely $5$ irreducible components.  One of them we know to be the closure of $\Rep_{\la_2} \SS(G_8)$; this closure is a component shared by all of the algebras $\la$ and $\la_i$.  The remaining components have generic radical layerings $\SStilde^{(1)} = (S_1, S_2, S_1, S_2)$, $\SStilde^{(2)} = (S_2, S_1, S_2, S_1)$, $\SStilde^{(3)} = (S_2, S_1^2, S_2, 0)$, and $\SStilde^{(4)} = (S_1^2, S_2^2, 0, 0)$, respectively.  We graph a generic module for each of the $\Rep_{\la_2} \SStilde^{(i)}$, $1 \le i \le 4$, in Figure \ref{ex9.8}(c).
\begin{figure}[ht]
$$\xymatrixrowsep{1.75pc} \xymatrixcolsep{0.8pc}
\xymatrix{
1 \dropvert5{\Rep_{\Lambda_2} \SStilde^{(1)}} \edge@/_2pc/[ddd]_{\alpha_2} \edge[d]^{\alpha_1} &&&&2 \dropvert5{\Rep_{\Lambda_2} \SStilde^{(2)}} \edge@/_2pc/[ddd]_{\beta_2} \edge[d]^(0.45){\beta_1} \edge@/^1pc/[dr]^{\beta_3} & &&& &2 \dropvert5{\Rep_{\Lambda_2} \SStilde^{(3)}} \edge[dl]_{\beta_1} \edge[dr]^{\beta_3} & &&&1 \edge@/_/[d]_{\alpha_1} \edge@/^/[d]^{\alpha_2} && \dropvert5{\Rep_{\Lambda_2} \SStilde^{(4)}} \save+<0ex,-4ex> \drop{\bigoplus} \restore&&1 \edge@/_/[d]_{\alpha_1} \edge@/^/[d]^{\alpha_2}  \\
2 \edge[d]^{\beta_1} &&&&1 
\save[0,0]+(0,3);[0,0]+(0,3) **\crv{~*=<2.5pt>{.} [0,0]+(3,3) &[0,1]+(0,3) &[0,1]+(3,3) &[0,1]+(3,0) &[2,1]+(3,0) &[2,1]+(3,-3) &[2,1]+(0,-3) &[2,0]+(0,-3) &[2,0]+(-3,-3) &[2,0]+(-3,0) &[2,0]+(-3,3) &[2,0]+(0,3) &[2,1]+(0,3) &[0,1]+(0,-3) &[0,0]+(0,-3) &[0,0]+(-3,-3) &[0,0]+(-3,0) &[0,0]+(-3,3)} \restore
\edge[d]^(0.6){\alpha_1} &1 &&&1 \edge[dr]^{\alpha_1} \edge@/_1pc/[dr]_{\alpha_2} &&1 \edge[dl]^{\alpha_1} \edge@/^1.5pc/[dl]^{\alpha_2} &&&2 &&&&2  \\
1 \edge[d]^{\alpha_1} \edge@/^1.5pc/[d]^{\alpha_2} &&&&2 \edge[d]^(0.4){\beta_1} & &&& &2   \\
2 &&&&1 &&
}$$
\centerline{Figure \ref{ex9.8}(c)}
\end{figure}

\noindent A confirmation of these claims, by means of the tools assembled in \ref{sec9}.B,C, is left to the reader. Moreover, $b = 4$ in the notation of Corollary \ref{cor9.7}.  In more detail: If $\SS^{(1)}, \dots, \SS^{(5)}$ are the generic radical layerings for the components of $\Rep_\bd(\la_0)$, sequenced as in Figure \ref{ex8.17}, then $\overline{\A_1}$ contains two components of $\Rep_\bd(\la_2)$, namely the closures of $\Rep_{\Lambda_2} \SStilde^{(1)}$ and $\Rep_{\Lambda_2} \SStilde^{(4)}$; further, $\overline{\A_2}$, $\overline{\A_4}$ contain one component of $\Rep_\bd(\la_2)$ each, while $\overline{\A_3}$ contains none. The fifth component of $\Rep_\bd(\la_2)$, finally, coincides with the closure of ${\A_5}$ in $\Rep_\bd(\la_2)$.  Note that all of the mentioned inclusions are proper (cf.~\ref{ex8.17}).
\end{example}

We conclude with an example of a monomial algebra $\Delta$ and a $\bd$-dimensional semisimple sequence $\SS$ with the property that $\Rep_{\Delta} \SS$ has two irreducible components, one of which closes off to an irreducible component of $\Rep_\bd(\Delta)$, whereas the other does not (cf.~the comments following Corollary \ref{cor9.4}). 

\begin{example} \label{ex9.9} Let $\Delta = KQ/\langle \beta_2 \alpha,\gamma_1\beta_2, \gamma_2\beta_1 \rangle$ be the monomial algebra based on the following quiver $Q$.  We include graphs of $\Delta e_1$, $\Delta e_2$ and $\Delta e_3$ for quick absorption of the ensuing argument.
$$\vcenter{\xymatrixcolsep{2pc}
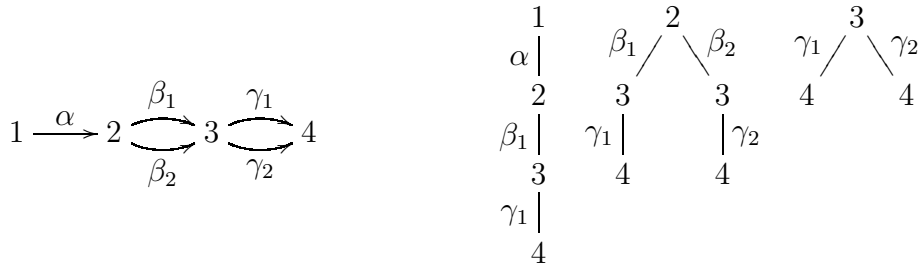
\xymatrix{
1 \ar[r]^{\alpha} &2 \ar@/^/[r]^{\beta_1} \ar@/_/[r]_{\beta_2} &3 \ar@/^/[r]^{\gamma_1} \ar@/_/[r]_{\gamma_2}
 &4 &&
}} 
\vcenter{\xymatrixrowsep{1.25pc} \xymatrixcolsep{0.5pc}
\xymatrix{
1 \edge[d]_{\alpha} &&&2 \edge[dl]_{\beta_1} \edge[dr]^{\beta_2} &&&&3  \edge[dl]_{\gamma_1} \edge[dr]^{\gamma_2} \\
2 \edge[d]_{\beta_1} &&3 \edge[d]_{\gamma_1} &&3 \edge[d]^{\gamma_2} &&4 &&4  \\
3 \edge[d]_{\gamma_1} &&4 &&4 \\
4}}
$$
\noindent Let $\bd = (1, 1, 1, 1)$ and $\SS = (S_1\oplus S_2, S_3, S_4, 0)$.  Then $\Rep_\Delta(\SS)$ has two irreducible components, $\D_1$ and $\D_2$, the orbits of the modules $G_1 = S_1 \oplus (\Delta e_2/ \Delta \beta_2)$ and $G_2 = S_1 \oplus (\Delta e_2/ \Delta \beta_1)$, respectively.   The closure of $\D_1$ fails to be an irreducible component of $\Rep_\bd(\Delta)$, since $G_1$ is a degeneration of $\Delta e_1$, whence $\D_1$ is contained in the closure of $\Rep_{\Delta} (S_1, S_2, S_3, S_4)$.  On the other hand, $\overline{\D_2}$ is an irreducible component of $\Rep_\bd(\Delta)$, since $G_2$ is the only left $\Delta$-module with dimension vector $\bd$ that has positive $\gamma_2 \beta_2$-rank.

Finally, observe that $\Rep_\bd(\Delta)$ has precisely $3$ irreducible components in total, namely $\overline{\Rep_{\Delta} (S_1, S_2, S_3, S_4)}$, $\, \overline{\D_2}$, and $\overline{\Rep_{\Delta_2} (S_1 \oplus S_3, S_2 \oplus S_4, 0, 0)}$. 
\end{example}



\begin{thebibliography}{88}

\bibitem{BHT} E. Babson, B. Huisgen-Zimmermann, and R. Thomas, \emph{Generic representation theory of quivers with relations}, J. Algebra \textbf{322} (2009), 1877--1918.

\bibitem{codes} \bysame, \emph{Maple codes for computing $\grassS$s}, posted at www.math.washington.edu/thomas/\\ programs/programs.html.

\bibitem{BaSch} M. Barot and J. Schr\"oer, \emph{Module varieties over canonical algebras}, J. Algebra \textbf{246} (2001), 175--192.

\bibitem{BCH} F. Bleher, T. Chinburg and B. Huisgen-Zimmermann, \emph{The geometry of algebras with vanishing radical square}, J. Algebra \textbf{425} (2015), 146--178. 

\bibitem{BCHII} F. Bleher, T. Chinburg and B. Huisgen-Zimmermann, \emph{The geometry of algebras with low Loewy length}, (in preparation). 

\bibitem{unifinIV} K. Bongartz and B. Huisgen-Zimmermann, \emph{Varieties of uniserial representations  IV. Kinship to geometric quotients}, Trans. Amer. Math. Soc. \textbf{353} (2001), 2091--2113 .

\bibitem{Car} A.T. Carroll, \emph{Generic modules for gentle algebras}, J. Algebra \textbf{437} (2015), 177--201.

\bibitem{CW} A.T. Carroll and J. Weyman, \emph{ Semi-invariants for gentle algebras}, Contemp. Math. \textbf{592} (2013), 111-136.

\bibitem{CB} W. Crawley-Boevey, \emph{Geometry of representations of algebras}, (1993), lectures posted at www1.maths.leeds.ac.uk/$\sim$pmtwc/geomreps.pdf.

\bibitem{CBS} W. Crawley-Boevey and J. Schr\"oer, \emph{Irreducible
components of varieties of modules}, J. reine angew. Math. \textbf{553} (2002), 201--220 .

\bibitem{dlP} J.A. de la Pe\~na, \emph{Tame algebras: Some fundamental notions}, Universit\"at Bielefeld, SFB 343, Preprint E95-010, (1995).

\bibitem{DeWe} H. Derksen and J. Weyman, \emph{On the canonical decomposition of quiver representations}, Compositio Math. \textbf{133} (2002), 245--265 .

\bibitem{DoFl} J. Donald and F.J. Flanigan, \emph{The geometry of Rep($A,V$) for a square-zero algebra}, Notices Amer. Math. Soc. \textbf{24} (1977), A-416.

\bibitem{Erd} K. Erdmann, \emph{Blocks of Tame Representation Type and Related Algebras}, Lecture Notes in Math. 1428, Berlin (1990), Springer-Verlag.

\bibitem{FrPe} E.M. Friedlander and J. Pevtsova, \emph{Representation theoretical support spaces for finite group schemes}, Amer. J. Math. \textbf{127} (2005), 379--420.

\bibitem{FPS} E.M. Friedlander, J. Pevtsova and A. Suslin, \emph{Generic and maximal Jordan type}, Invent. Math. \textbf{168} (2007), 485--522.

\bibitem{GeiSchI} Ch. Geiss and J. Schr\"oer, \emph{Varieties of modules over tubular algebras}, Colloq. Math. \textbf{95} (2003), 163--183.

\bibitem{GeiSchII} \bysame, \emph{Extension-orthogonal components of preprojective varieties}, Trans. Amer. Math. Soc. \textbf{357} (2004), 1953--1962.

\bibitem{GePo} I.M. Gelfand and V.A. Ponomarev, \emph{Indecomposable representations of the Lorentz group}, Usp. Mat. Nauk \textbf{23} (1968), 3-60; Engl. transl.:  Russian Math. Surv. \textbf{23} (1968), 1--58.

\bibitem{GHZ} K.R. Goodearl and B. Huisgen-Zimmermann, \emph{Closures in varieties of representations and irreducible components}, Algebra and Number Theory, to appear.

\bibitem{classifying} B. Huisgen-Zimmermann, \emph{Classifying representations by way of Grassmannians}, Trans. Amer. Math. Soc. \textbf{359} (2007), 2687--2719.

\bibitem{hier} \bysame, \emph{A hierarchy of parametrizing varieties for representations}, in \emph{Rings, Modules and Representations} (N.V. Dung, et al., eds.), Contemp. Math. \textbf{480} (2009), 207--239.

\bibitem{irredcompI} \bysame, \emph{Irreducible components of varieties of representations.  The local case}, J. Algebra \textbf{464} (2016), 198--225.

\bibitem{irredcompII} B. Huisgen-Zimmermann and I. Shipman, \emph{Irreducible components of varieties of representations.  The acyclic case}, Math. Zeitschr. \textbf{287} (2017), 1083--1107.

\bibitem{JeLe} C.U. Jensen and H. Lenzing, \emph{Model Theoretic Algebra}, New York (1989) Gordon and Breach.

\bibitem{KacI} V. Kac, \emph{Infinite root systems, representations of
graphs and invariant theory}, Invent. Math. \textbf{56} (1980),
57--92.

\bibitem{KacII} \bysame, \emph{Infinite root systems, representations of
graphs and invariant theory, II}, J. Algebra \textbf{78} (1982),
141--162 .

\bibitem{KaSa} M. Kashiwara and Y. Saito, \emph{Geometric construction of crystal bases}, Duke Math. J. \textbf{89} (1997), 9--36.

\bibitem{MHSa} F.H. Membrillo-Hern\'andez and L. Salmer\'on, \emph{A geometric approach to the finitistic dimension conjecture}, Archiv Math. \textbf{67} (1996), 448--456.

\bibitem{Mor} K. Morrison, \emph{The scheme of finite-dimensional representations of an algebra}, Pac. J. Math. \textbf{91} (1980), 199--218.

\bibitem{Rich} N.J. Richmond, \emph{A stratification for varieties of modules}, Bull. London Math. Soc. \textbf{33} (2001), 565--577.

\bibitem{RiRuSm} C. Riedtmann, M. Rutscho, and S.O. Smal\o, \emph{Irreducible components of module varieties: An example}, J. Algebra \textbf{331} (2011), 130--144
.
 
 \bibitem{RinI} C.M. Ringel, \emph{Tame Algebras and Integral Quadratic Forms}, Lecture Notes in Math. 1099, Berlin (1984) Springer-Verlag.
 
\bibitem{Rin} C.M. Ringel, \emph{The preprojective algebra of a quiver}, in \emph{Algebras and Modules II (Geiranger,
1996)}, CMS Conf. Proc. \textbf{24} (1998), 467--480.
 
\bibitem{Ros} M. Rosenlicht, \emph{A remark on quotient spaces}, An. Acad. Brasil. Ci. \textbf{35} (1963), 487--489.
 
\bibitem{Scho} A. Schofield, \emph{General representations of quivers}, Proc. London Math. Soc. (3) \textbf{65} (1992), 46--64 .

\bibitem{Schro} J. Schr\"oer, \emph{Varieties of pairs of nilpotent matrices annihilating each other}, Comment. Math. Helv. \textbf{79} (2004), 396--426 .

\end{thebibliography}
\end{document}